\documentclass[12pt,a4paper]{article}

\usepackage{amssymb, amsmath, amsthm}

\usepackage[cm]{fullpage}
\usepackage[english]{babel}
\usepackage[pdftex]{graphicx}
\usepackage{booktabs}
\usepackage{xcolor}
\usepackage{hyperref}

\newtheorem{thm}{Theorem}
\newtheorem{lem}{Lemma}
\newtheorem{prop}{Proposition}
\newtheorem{rem}{Remark}
\newtheorem{cor}{Corollary} 

\newcommand{\veps}{\varepsilon}

\title{Convolution Quadrature for the quasilinear subdiffusion equation\footnote{This is the accepted version of the manuscript to appear in SIAM Journal on Numerical Analysis.}}
\date{}
\author{Maria L\'opez-Fern\'andez\thanks{Department of Mathematical Analysis, Statistics and O.R., and Applied Mathematics, Faculty of Sciences, University of Malaga, Bulevar Louis Pasteur, 31 29010 Malaga, Spain, email: maria.lopezf@uma.es}, \and \L ukasz P\l ociniczak\thanks{Faculty of Pure and Applied Mathematics, Wroc{\l}aw University of Science and Technology, Wyb. Wyspia{\'n}skiego 27, 50-370 Wroc{\l}aw, Poland, email: lukasz.plociniczak@pwr.edu.pl}}

\begin{document}
\maketitle

\begin{abstract}
	We construct a Convolution Quadrature (CQ) scheme for the quasilinear subdiffusion equation of order $\alpha$ and supply it with the fast and oblivious implementation. In particular, we find a condition for the CQ to be admissible and discretize the spatial part of the equation with the Finite Element Method. We prove the unconditional stability and convergence of the scheme and find a bound on the error. Our estimates are globally optimal for all $0<\alpha<1$ and pointwise for $\alpha\geq 1/2$ in the sense that they reduce to the well-known results for the linear equation. For the semilinear case, our estimates are optimal both globally and locally.  As a passing result, we also obtain a discrete Gr\"onwall inequality for the CQ, which is a crucial ingredient in our convergence proof based on the energy method. The paper is concluded with numerical examples verifying convergence and computation time reduction when using fast and oblivious quadrature. \\	
	
\end{abstract}

\section{Introduction}
Consider the following quasilinear subdiffusion equation with vanishing Dirichlet condition on a smooth domain $\Omega\subseteq\mathbb{R}^d$ with $d\in\mathbb{N}$
\begin{equation}\label{eqn:MainPDE}
	\begin{cases}
		\partial^\alpha_t u = \nabla\cdot \left(D(x,t,u) \nabla u\right) + f(x,t,u), & x\in \Omega, \; t>0, \; \alpha\in (0,1),\\
		u(x,0) = 0, & x\in \Omega, \\
		u(x,t) = 0, & x \in \partial\Omega,
	\end{cases}
\end{equation}
where $\partial^\alpha_t$ is the partial Caputo time derivative 
\begin{equation}\label{eqn:Caputo}
	\partial_t^\alpha u(x, t) := I^{1-\alpha} \frac{\partial}{\partial t} u(x, t) = \frac{1}{\Gamma(1-\alpha)}\int_0^t (t-s)^{-\alpha} \frac{\partial u}{\partial t}(x,s) ds, \quad 0<\alpha<1,
\end{equation}
defined with the help of the fractional integral
\begin{equation}\label{eqn:FracInt}
	I^\alpha u(x, t) := \frac{1}{\Gamma(\alpha)} \int_0^t (t-s)^{\alpha-1} u(x, s) ds, \quad \alpha > 0,
\end{equation} 
Note that the vanishing of the initial condition in \eqref{eqn:MainPDE} can be assumed without any loss of generality. To wit, assume that $u(x,0) = u_0(x) \in H_0(\Omega)\cap H^2(\Omega)$, then by introducing $v = u - u_0$ we can easily show that $v$ satisfies \eqref{eqn:MainPDE} with a new diffusivity $\widetilde{D}(x,t,v) = D(x,t,v+u_0)$ and a new source $\widetilde{f}(x,t,v) = f(x,t,v+u_0) + \nabla\cdot(D(x,t,v+u_0) \nabla u_0)$. Therefore, in what follows, we will consider the general case \eqref{eqn:MainPDE}.

Our assumptions on the regularity of the coefficients are as follows. Let $D\in C^2(\Omega, \mathbb{R}_+, \mathbb{R})$ and $f\in C^2(\Omega, \mathbb{R}_+, \mathbb{R})$. Moreover, we assume that $D$ is bounded from below and above, and $L>0$ is the Lipschitz constant for the third argument of $f$ and $D$, that is
\begin{equation}\label{eqn:Assumptions}
	0<D_-\leq D(x,t,u) \leq D_+, \quad |D(x,t,u)-D(x,t,v)| + |f(x,t,u) - f(x,t,v)| \leq L |u-v|,
\end{equation}
and this guarantees well-posedness of the problem.  However, even with weaker conditions, it has been proven in \cite{Zac12} that \eqref{eqn:MainPDE} has a unique strong solution. More specifically, for a $C^2$-smooth domain $\Omega$ and with $p>d+2/\alpha$ we have $u\in W^{\alpha,p}([0,T];L^p(\Omega)) \cap L^p\left([0,T];W^{2,p}(\Omega)\right)$. Additional solvability results can be found in \cite{akagi2019fractional}. Furthermore, large-time decay estimates have been established in \cite{Ver15,Dip19}. On the other hand, viscosity solutions to \eqref{eqn:MainPDE} have been studied in \cite{Top17}. In particular it has been shown that $|u(x,t)| \leq C t^\alpha$ for $t\in [0,T]$, suggesting $\alpha$-H\"older continuity at the time origin. Some additional results that are also valid in the degenerate case when diffusivity can vanish were proved in the weak setting in \cite{akagi2019fractional,Wit21,All16}. The semilinear constant coefficient case, that is when $D=$const. and $f=f(x,t,u)$, has been investigated, for example, in \cite{al2019numerical} where $\alpha$-H\"older continuity of the solution was established under sufficient regularity conditions on the source. The linear case is very well understood in the constant and $x$-dependent diffusivity. Details of the solution can be found in \cite{sakamoto2011initial,kubica2020time}. The linear case with time-dependent diffusivity has been carefully studied in \cite{JinLiZhou2019}, where a detailed regularity analysis is provided and error estimates are derived for a scheme based on BDF1-CQ, and in \cite{JinLiZhou2020}, where a corrected BDF2-CQ scheme is proposed and proven to achieve the second order of convergence. After submitting this paper, we have become aware of the reference \cite{JinQuanWohlZhou24}, where a careful regularity analysis for the quasilinear equation \eqref{eqn:MainPDE} with a linear source $f(x,t,u)=f(t)$ is presented and its discretization is analyzed using a corrected BDF (2)-CQ scheme, under more stringent regularity assumptions on the data than we make here and by using a very different analysis technique.

The most spectacular difference between a solution to classical diffusion and its slower, subdiffusive version is the amount of smoothing of the initial data. To be precise, it is known that for the linear PDE with $D=$ const. and $f\in C^{m-1}([0,T];L^2(\Omega))$ with $I_t^\alpha(\|\partial^{(m)}_t f(t)\|) < \infty$ we have \cite{jin2019numerical} (Theorem 2.1, (iii))
\begin{equation}
	\|u^{(m)}(t)\| \leq C \sum_{k=0}^{m-1} t^{\alpha-(m-k)} \|\partial^{(k)}_t f(0)\| + I_t^\alpha(\|\partial^{(m)}_t f(t)\|),
\end{equation}
where the superscript denotes the time derivative of the solution regarded as a mapping from $[0,T]$ into the $L^2(\Omega)$ space. This means that even for very smooth $f$, the solution can still be of limited regularity at $t=0$ unless sufficiently many time derivatives of $f$ initially vanish. It is also known from \cite{al2019numerical} that the solution to the semilinear equation with an initial condition $u_0\in H^{\nu}$ with $\nu\in(0,2]$ satisfies the following regularity estimates
\begin{equation}
	\label{eqn:Regularity}
	\begin{cases}
		u \in C^\frac{\alpha\nu}{2}([0,T]; L^2(\Omega)) \cap C([0,T]; H^\nu(\Omega)) \cap C((0,T]; H^2(\Omega)), &\\
		\|\partial^{(m)}_t u(t)\| \leq C(u) t^{\alpha-m}, & m = 0,1,
	\end{cases}
\end{equation}
Although the solution is continuous on $[0,T]$ it has a singular time derivative, and we have to expect that for our quasilinear problem the situation can be at most as regular as above. Therefore, in what follows, we assume that our solution has the following regularity
\begin{equation}
	\label{eqn:RegularityAssumtion}
	u \in C^2((0,T]; H_0^1(\Omega) \cap H^2(\Omega)), \quad\|\partial^{(m)}_t u
	(t)\| \leq C(u) t^{\alpha-m} \text{ for } m = 0, 1, 2, \quad \|\nabla u(t) \|_\infty \leq C.
\end{equation}
The literature on CQ schemes for semilinear equations contains some interesting results. In \cite{jin2018numerical} authors established a global Euler-CQ $\alpha$-order of convergence in time assuming smooth initial data. This result was improved and generalized in \cite{al2019numerical} to nonsmooth initial data $u_0\in \dot{H}^\nu(\Omega)$, with $\nu\in (0,2]$, yielding order $1$ in time locally and $\alpha \nu /2$ globally. This is in agreement with our numerical computations for the quasilinear equation presented below (for smooth initial data) and also corresponds closely with the linear case. A higher order CQ linearised scheme was developed in \cite{wang2020high}. Authors proved that even for smooth initial data and high order BDF, the scheme has a limited order of convergence. In the quasilinear case we can expect that results obtained will be at most as the ones in the semilinear version of the subdiffusion equation.

The governing equation \eqref{eqn:MainPDE} arises in many areas of science as a model of subdiffusive phenomena. Generally speaking, \emph{subdiffusion} denotes a slower than usual random motion of a collection of particles, in contrast to the classical diffusion and faster evolution known as superdiffusion. To be more precise, if we consider a randomly moving particle with a mean-squared displacement proportional to $t^\alpha$, then we say that it is classically dispersing when $\alpha=1$. Sub- and superdiffusive evolution occurs when $0<\alpha<1$ and $1<\alpha<2$, respectively \cite{Met00,Kla12}. Another important application of \eqref{eqn:MainPDE} arises in hydrology when considering moisture percolation inside a porous medium \cite{El20,kuntz2001experimental}. A derivation of our governing equation in the hydrological setting has been given in \cite{Plo15}, where it has been shown that the slower than classical evolution can be a consequence of the fluid being trapped in some regions of the porous medium. This can be the result of nonhomogeneity or chemical reactions that take place in the domain \cite{El20}. Note that it has been observed that the evolution of moisture within a porous medium necessarily has to be described by a nonlinear equation, since the diffusivity can change by orders of magnitude when the pores are filled with water \cite{bear2013dynamics}. Other important applications of the subdiffusion equation can be found, for example, in: biology \cite{Sun17}, finance \cite{magdziarz2009black}, and chemistry \cite{weeks2002subdiffusion} to name only a few examples.  

Fractional differential equations have been extensively investigated both analytically and numerically. In order not to go too far in reviewing the previous results, we will focus only on numerical methods for the subdiffusion equation in its various forms. Several numerical schemes have been devised to study \eqref{eqn:MainPDE}. As mentioned above, the L1 scheme was applied in \cite{plociniczak2023linear} along with convergence proofs. An interesting account of an even more general problem - including a stochastic term - has been investigated numerically in \cite{Liu18}. To our knowledge, this is just the beginning of rigorous numerical analysis of the quasilinear subdiffusion equation, and several authors are making progress in this field. For the time-fractional parabolic PDEs of a simpler form, one can also find many interesting results. For example, a semilinear equation with constant diffusivity has been discretized with the backward Euler scheme in time and the FEM in space in \cite{al2019numerical} and with higher-order convolution quadratures in \cite{li2022exponential}. In these papers, the authors allowed for nonsmooth initial data, which is a realistic and more difficult case. The numerical analysis of this problem was later expanded to include the variable in space and time diffusivity, with a linear source in \cite{JinLiZhou2019,Mus18}. Lately, optimal-order estimates for the semidiscrete Galerkin numerical method with nonsmooth data and the fully general but semilinear subdiffusion equation have been obtained in \cite{plociniczak2022error} under weak assumptions. Finally, we mention a few notable papers that introduced and analyzed various numerical methods for purely linear equations with the Caputo time derivative. In \cite{jin2016two} two fully discrete schemes based on modified convolution quadrature have been developed and have been shown to achieve the optimal order of convergence with respect to the smoothness of the initial data. The L1 method has been utilized, for example, in \cite{kopteva2019error, Lia18, stynes2017error}, where optimal order estimates have also been given even for nonuniform grids.

We discretize \eqref{eqn:MainPDE} in space by applying the Finite Element Method with piecewise linear elements and consider for the temporal approximation of the semidiscrete problem a semi-implicit scheme where the fractional derivative is approximated by Lubich's Convolution Quadrature (CQ) method \cite{Lu88I}. We derive sufficient conditions for the CQ that guarantee the stability and convergence of the resulting scheme. As a side result, we are able to obtain a new version of the Gr\"onwall's inequality that is suitable for use in the context of admissible convolution quadratures. To the best of our knowledge, this is the first CQ approach to discretization of the time-fractional quasilinear diffusion equation. A crucial point in our convergence proof is based on the aforementioned Gr\"onwall's lemma and a new coercivity result for the CQ methods (other results concerning a different approach to CQ coercivity can be found in \cite{banjai2019runge} and in the monograph \cite{banjai2022integral}, Section 2.6). Thanks to these, the quasilinear case can be analyzed via the energy method as opposed with previous operator approaches. Therefore, we are able to present a rigorous analysis of a nonlinear subdiffusion equation based on a CQ scheme that allows for a fast and oblivious implementation. Indeed, by applying the algorithm in \cite{BanLo19}, the memory requirements can be reduced from the $O(N)$  required by a straightforward implementation of the CQ, see \cite{Lu88II}, to $O(\log(N))$, being $N$ the total number of time steps, and the complexity can be reduced from $O(N^2)$ to $O(N\log(N))$. For BDF$(p)$ quadratures, the order of convergence of our method remains globally optimal, having the same form as in linear subdiffusion equations, that is, equal to $\alpha$. For $\alpha\geq 1/2$ we are also able to prove pointwise optimal estimates. That is, for the BDF$(p)$ quadrature the error in time behaves as $t_n^{\alpha-1} h$, where $h$ is the time step (for the Euler scheme we also observe a slowly increasing logarithmic factor). We also prove that for the semilinear equation our estimates are optimal both globally and locally.  Naturally, away from $t=0$ the method exhibits the first order of convergence which deteriorates to order $\alpha$ when $t\rightarrow 0^+$. This behavior is typical for linear equations \cite{cuesta2006convolution,jin2016two}, also for formulas of higher order \cite{JinLiZhou2017}, and it is visible in the L1 method discretization \cite{kopteva2019error}, too. We are able to extend this result to quasilinear equations with minimal regularity assumptions on the solutions.  

The paper is organized as follows. In Section 2 we prove a coercivity result for the CQ discretization of the fractional derivative and a suitable version of Gr\"onwall's lemma for the analysis of \eqref{eqn:MainPDE}. In Section 3 we present our numerical scheme and provide complete error estimates in both time and space. Section 4 describes the implementation in time of our method and the application of the fast and oblivious algorithm from \cite{BanLo19} in this setting. The numerical results confirming our theoretical results are shown in Section 5, together with some comparisons of complexity requirements with other schemes in the recent literature.

\section{Properties of convolution quadratures for the Caputo derivative}
We will start by discussing some properties of CQ that will be useful for the energy method. Fix the uniform time mesh $0\leq t_n := n h \leq T$, with step $h>0$, and a function $y(t)$. Provided that the initial condition vanishes, that is $y(0) = 0$, the CQ approximation of the Caputo derivative \eqref{eqn:Caputo} is given by
\begin{equation}
	\label{eqn:CQCaputo}
	\partial_h^\alpha y^n = \sum_{j=0}^n w_{n-j}y(t_j) \mbox{ with the CQ weights given by } \sum_{j=0}^\infty w_{j} \zeta^j = \left(\frac{\delta(\zeta)}{h}\right)^{\alpha},
\end{equation}
with $\delta(\zeta)$ the symbol of the underlying ODE solver. For the implicit Euler-based CQ it is $\delta(\zeta)=1-\zeta$, while for the CQ based on the BDF2 method it is $\delta(\zeta)=1-\zeta + \frac 12 (1-\zeta)^2$. CQ based on high--order Runge--Kutta methods are also available \cite{LuOs}, but will not be considered in the present work. The notation in \eqref{eqn:CQCaputo} for the CQ discretization extends in a straightforward way to vectors $v=(v^j)_{j=1}^N$, so that  we will also denote $\partial_h^{\alpha}v$ the vector with components given by
\begin{equation}\label{CQvec}
	\partial_h^{\alpha}v^n = \sum_{j=0}^n w_{n-j}v^j, \quad n=0,\dots,N.
\end{equation}
We can also define the corresponding quadrature for the fractional integral \eqref{eqn:FracInt} with the same symbol as in \eqref{eqn:CQCaputo}, that is
\begin{equation}\label{eqn:CQFracInt}
	I^\alpha y(t_n) \approx \sum_{j=0}^n b_{n-j}y(t_j), \quad \sum_{j=0}^\infty b_{j} \zeta^j = \left(\frac{\delta(\zeta)}{h}\right)^{-\alpha}.
\end{equation}
In what follows we will always assume that the weights have the following signs
\begin{equation}\tag{A}
	\label{eqn:CQCaputoWeights}
	w_{0} := h^{-\alpha} \delta(0)^{\alpha} > 0, \quad w_{j} < 0 \quad \text{for} \quad j\geq 1.
\end{equation}
This assumption is motivated by our subsequent results in this section. The simplest example of the above condition is the Backward Euler scheme for which we have $\delta(\zeta) = 1-\zeta$. From the binomial series we have the weights
\begin{equation}
	\begin{split}
		w_0 &= h^{-\alpha} > 0, \\
		w_j &= (-1)^j\binom{\alpha}{j} h^{-\alpha}= (-1)^j \frac{\alpha(\alpha-1)...(\alpha-j+1)}{j!} = - \frac{\alpha(1-\alpha)...(j-1-\alpha)}{j!} h^{-\alpha}< 0,
	\end{split}
\end{equation}
where in the last equality we have taken the minus sign from all $j$ factors canceling the $(-1)^j$ term. The general sign follows because all parentheses are positive for $\alpha \in (0,1)$. For the BDF2 weights can also be written explicitly.
\begin{prop}
	Let $w_j$ given by \eqref{eqn:CQCaputo} be the weights associated with the BDF2 formula. Then,
	\begin{equation}
		w_j = (-1)^j 2^{-\alpha} 3^{\alpha-j} \binom{\alpha}{j}\,_2 F_1\left(-j,-\alpha;1-j+\alpha;3\right) h^{-\alpha},
	\end{equation}
	where the hypergeometric function is defined by
	\begin{equation}
		_2 F_1(a,b,c; z) = \sum_{k=0}^\infty \frac{(a)_k (b)_k}{(c)_k} \frac{z^k}{k!},
	\end{equation}
	with the Pochhammer symbol $(a)_k = a(a+1)...(a+k-1)$. 
\end{prop}
\begin{proof}
	The symbol for BDF2 can be factored as $\delta(\zeta) = 1-\zeta + (1-\zeta)^2/2 = 0.5(1-\zeta)(3-\zeta)$. Therefore, by the definition of the weights \eqref{eqn:CQCaputo} we have
	\begin{equation}
		\sum_{j=0}^\infty w_j \zeta^j = \left(\frac{\delta(\zeta)}{h}\right)^{\alpha} = 2^{-\alpha} 3^\alpha h^{-\alpha} \sum_{j=0}^\infty \left(\sum_{k=0}^j \binom{\alpha}{k} \binom{\alpha}{j-k}3^{-k}\right) \zeta^j,
	\end{equation} 
	where we have used the Cauchy product formula for Taylor series for $(1-\zeta)^\alpha$ and $(1-\zeta/3)^\alpha$. Therefore, it is sufficient to evaluate the inner sum. First, notice that
	\begin{equation}
		k! \binom{\alpha}{k} = \alpha(\alpha-1)...(\alpha-k+1) = (-1)^k (-\alpha)(-\alpha+1)...(-\alpha+k-1) = (-1)^k (-\alpha)_k. 
	\end{equation}
	Therefore, by the definition of the binomial coefficient,
	\begin{equation}
		\begin{split}
			w_j &= 2^{-\alpha} 3^\alpha h^{-\alpha} \sum_{k=0}^j (-1)^k (-\alpha)_k \binom{\alpha}{j-k}\frac{3^{-k}}{k!} \\
			&= 2^{-\alpha} 3^\alpha h^{-\alpha} \binom{\alpha}{j} \sum_{k=0}^j (-1)^k (-\alpha)_k \frac{j!}{(j-k)!}\frac{\Gamma(\alpha+1-j)}{\Gamma(\alpha+1-j+k)}\frac{3^{-k}}{k!}.
		\end{split}
	\end{equation}
	The above is precisely the definition of the hypergeometric function and this can be seen by noticing that
	\begin{equation}
		\frac{j!}{(j-k)!} = j(j-1)...(j-k+1) = (-1)^k (-j)_k,
	\end{equation}
	and 
	\begin{equation}
		\frac{\Gamma(\alpha+1-j)}{\Gamma(\alpha+1-j+k)} = \frac{1}{(\alpha+1-j)...(\alpha-j+k)} = \frac{1}{(\alpha+1-j)_k}.
	\end{equation}
	The proof is complete since due to the factor $(-j)_k$, the series terminates after $k=j$. 
\end{proof}

It is not straightforward to prove that the BDF2 weights satisfy the condition \eqref{eqn:CQCaputoWeights}, however, we can easily see that
\begin{equation}\label{eqn:BDF2Weights}
	w_1 = -2^{2-\alpha} 3^{\alpha-1} \alpha h^{-\alpha}, w_2 = -2^{-\alpha} 3^{-2+\alpha} (5-8\alpha) h^{-\alpha}, \quad w_3 = -2^{2-\alpha} 3^{4-\alpha} \alpha(1-\alpha) (7-8\alpha) h^{-\alpha},
\end{equation}
and hence, the first weight is negative, the second one only for $0<\alpha<5/8=0.625$, while the third one for $0<\alpha<7/8=0.875$. We have numerically checked the sign of a number of subsequent weights and confirmed that all of them satisfy our assumption \eqref{eqn:CQCaputoWeights}. That is, all BDF2 weights are admissible for $0<\alpha<5/8$. This can also be visualized numerically. In Fig. \ref{fig:BDF2} we have plotted the respective weights $w_j$ for all $0<\alpha<1$. Computations confirm that $|w_j|$ are also $j-$decreasing, as can be seen in Fig. \ref{fig:BDF2j}. It can be seen that the numerical calculations confirm our hypothesis. 

\begin{figure}
	\centering
	\includegraphics[scale = 1.2]{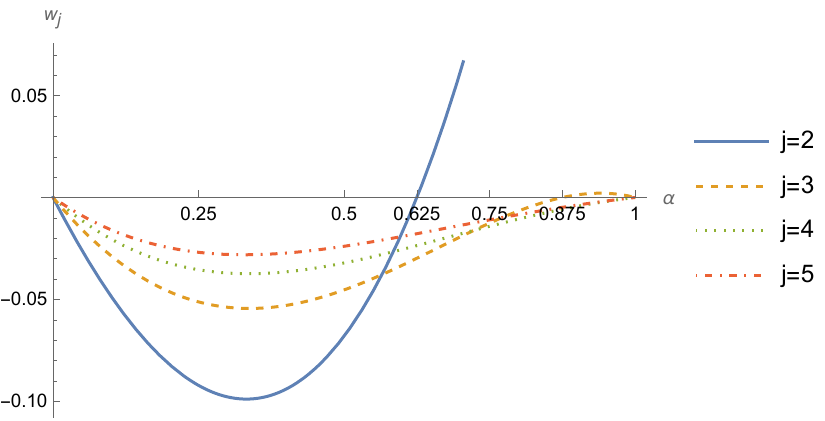}
	\caption{Plot of the BDF(2) weights $w_j$ as in \eqref{eqn:BDF2Weights} for different $0<\alpha<1$ with $h=1$. Weight $w_2$ is negative for $0<\alpha<5/8$, weight $w_3$ is negative for $0<\alpha<7/8$, while all other weights are always negative.  }
	\label{fig:BDF2}
\end{figure}

\begin{figure}
	\centering
	\includegraphics[scale = 1.2]{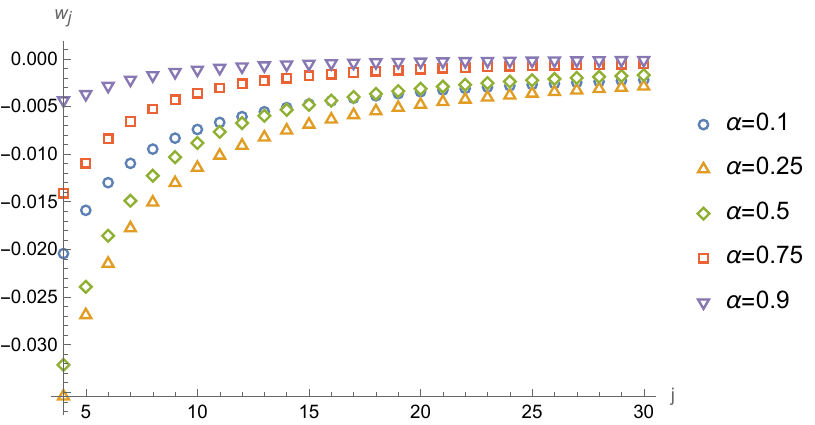}
	\caption{Plot of the BDF(2) weights $w_j$ as in \eqref{eqn:BDF2Weights} for different $j\geq 4$ with $h=1$. }
	\label{fig:BDF2j}
\end{figure}

\begin{rem}
	We note that obtaining exact explicit formulas for the weights for the general CQ quadratures can be impossible, but only proving that they have a sign satisfying \eqref{eqn:CQCaputoWeights}. However, according to the general theory of CQ we have the following (see \cite{lubich2004convolution}, Theorem 2.1)
	\begin{equation}
		h^{-1}w_j = \frac{1}{\Gamma(1-\alpha)}\frac{d}{dt} t^{-\alpha}|_{t=t_j} + O(h^p) = -\frac{\alpha}{\Gamma(1-\alpha)} t_j^{-\alpha-1} + O(h^p),
	\end{equation}
	for $p\geq 1$ depending on the order of the CQ formula. From the above we can see that for $j$ sufficiently large and $h$ small, the sign of the weights is the same as the sign of $-\alpha/\Gamma(1-\alpha)$, that is, negative. Note that this conclusion is not necessarily valid for small $j$ which is precisely the case for BDF2 weights. 
\end{rem}

Going back to the general case, by the very construction of the approximation to $\partial_t^\alpha$ we have the consistency condition
\begin{equation}
	\label{eqn:CQCaputoConsistency}
	\sum_{j=0}^\infty w_j = 0 \quad n\geq 1,
\end{equation}
which follows by putting $\zeta=1$ into (\ref{eqn:CQCaputo}) or by requiring that any CQ scheme for the Caputo derivative is exact for constant functions (and hence, identically equals zero). From this it follows that for any $n\geq 1$ we have
\begin{equation}
	\label{eqn:CQCaputoConsistencyN}
	\sum_{j=0}^n w_j = \sum_{j=0}^{\infty} w_j - \sum_{j=n+1}^\infty w_j \geq 0,
\end{equation}
by the assumption \eqref{eqn:CQCaputoWeights} that the weights $w_j$ are negative for $j\ge 1$. We also have the truncation error
\begin{equation}
	\label{eqn:CQCaputoTruncation}
	\partial_t^\alpha u(x, t_n) = \partial_h^\alpha u(x, t_n) + \zeta_n(h).
\end{equation}
The term $\zeta_n(h)$ can be estimated with the help of \cite{lubich2004convolution}, Theorem 2.2, see also \cite[Theorem 5.1-5.2]{Lu88I}. Under the assumption that the solution has the typical regularity near the origin, that is $\bar{u}(t):=u(x,t)$,  with $\bar{u}(t)=t^{\alpha}g(t)$, $g\in C^{p+1}\left([0,T]; H_0^1(\Omega) \cap H^2(\Omega) \right)$ (see our assumption \eqref{eqn:RegularityAssumtion}) we have 
\begin{equation}\label{eqn:CQCaputoTruncationEstimate}
	\|\zeta_n(h)\| \leq C \begin{cases}
		t_n^{-1} h, & p = 1, \\
		t_n^{-\alpha-1} h^{\alpha+1}, & p \geq 2. 
	\end{cases}
\end{equation}
This clearly states how the error deteriorates near the origin due to the lack of smoothness of the solution. 

We will now prove two auxiliary results that are discrete generalizations of known continuous inequalities for the Caputo derivative. They will be used later in the following section, but they are also interesting on their own. First, it is clear that the first ordinary derivative satisfies $\frac{1}{2}\frac{d}{dt}y(t)^2 = y(t) \frac{dy}{dt}$. For the continuous Caputo derivative, it becomes an inequality $\frac{1}{2}D^\alpha y(t)^2 \leq y(t) D^\alpha y(t)$ which is very useful when applied in the energy method (see \cite{alikhanov2010priori,Ver15,Kop22} for recent proofs in the case where $y=y(t)\in L^2$ is a time-dependent mapping from the Hilbert space). As the following proposition states, this inequality is still valid for the convolution quadrature constructed above (the L1 scheme version has recently been proved in \cite{li2018analysis}, while the case of Euler CQ in \cite{jin2021numerical}). 

\begin{prop}\label{prop:CQInner}
	Let the weights $w_j$ for the convolution quadrature (\ref{eqn:CQCaputo}) satisfy (\ref{eqn:CQCaputoWeights}). Then, for any sequence of functions $(y^n)_n \subseteq L^2$ with $y(0) = 0$ we have 
	\begin{equation}
		\frac{1}{2}\partial_h^\alpha \|y^n\|^2 \leq \|y^n\| (\partial_h^\alpha \|y^n\|) \leq (\partial_h^\alpha y^n, y^n).
	\end{equation}
\end{prop}
\begin{proof}
	Using the definition of weights (\ref{eqn:CQCaputo}) we can write 
	\begin{equation}
		(\partial_h^\alpha y^n, y^n) = w_{0} \|y^n\|^2 + \sum_{j=0}^{n-1} w_{n-j} (y^j, y^n).
	\end{equation}
	Since by assumption \eqref{eqn:CQCaputoWeights} the negativity of $w_{j}$ for $j\geq 1$ holds, we can use the Cauchy-Schwarz inequality to obtain
	\begin{equation}
		(\partial_h^\alpha y^n, y^n) \geq w_{0} \|y^n\|^2 + \|y^n\|\sum_{j=0}^{n-1} w_{n-j} \|y^j\| = \|y^n\|(\partial_h^\alpha \|y^n\|),
	\end{equation}
	which is the first inequality that we had to prove. Furthermore, by the Cauchy inequality $ab \leq (a^2+b^2)/2$ we can estimate each product with $\|y^n\|$ by the sum of squares
	\begin{equation}
		\|y^n\|(\partial_h^\alpha \|y^n\|) \geq \frac{1}{2}\left(w_{0} + \sum_{j=0}^{n-1} w_{n-j} \right) \|y^n\|^2 + \frac{1}{2}\left(\|y^n\|^2+\sum_{j=0}^{n-1} w_{n-j} \|y^j\|^2\right),
	\end{equation}
	where we have distributed $w_0$ into two half for each parenthesis. By gathering terms, we obtain
	\begin{equation}
		\|y^n\|(\partial_h^\alpha \|y^n\|) \geq \frac{1}{2}\|y^n\|^2\sum_{j=0}^{n} w_j  + \frac{1}{2}\sum_{j=0}^{n} w_{n-j} \|y^j\|^2.
	\end{equation}
	But, according to the consistency (\ref{eqn:CQCaputoConsistencyN}) the first term is positive leading us to the assertion. 
\end{proof}

The next result concerns the discrete Gr\"onwall inequality for the Caputo derivative. Its version for the L1 discretization has been proved, for example, in \cite{li2018analysis,Lia18}. The proof is based on two steps: first, we invert the derivative and then use the following classical lemma, which is a generalization of the discrete integral Gr\"onwall inequality (historically, it was not associated with the name "fractional integral" but rather "integral equations with weakly-singular kernel"). 
\begin{prop}[Discrete fractional Gr\"onwall inequality (integral version) (\cite{dixon1985order}, Theorem 2.1)]
	\label{prop:GronwallInt}
	Let $(y^n)_n$ be a positive sequence satisfying
	\begin{equation}
		y^n \leq M_0 t_n^{\alpha-1} + M_1 h^{\alpha}\sum_{k=0}^{n-1} (n-k)^{\alpha-1} y^k + M_2,
	\end{equation}
	for some positive constants $M_{1,2,3}$, which may depend on $h$. Then, for $0<\alpha<1$ we have
	\begin{equation}
		y^n \leq M_0\Gamma(\alpha) t_n^{\alpha-1} E_{\alpha,\alpha} (M_1 \Gamma(\alpha) t_n^\alpha) + M_2 E_\alpha(M_1 \Gamma(\alpha) t_n^\alpha),
	\end{equation}
	where the Mittag-Leffler function is defined by
	\begin{equation}
		E_{\alpha,\beta}(z) = \sum_{k=0}^\infty \frac{z^k}{\Gamma(\alpha k + \beta)}, \quad E_\alpha := E_{\alpha,1}.  
	\end{equation}
\end{prop}

Now, we will state the first step in proving the needed version of the Gr\"onwall lemma. It states that the inequality involving a discrete Caputo derivative for CQ can be ,,integrated''. Some similar reasoning for applying the mathematical induction for the Caputo derivative has been used in \cite{jin2017analysis} for the L1 scheme. The lemma below follows directly from the preservation of the composition rule by all CQ formulas, without any assumption about the signs of the CQ weights.

\begin{lem}\label{lem:CaputoIntegrated}
	Let $(y^j)_{j=0}^\infty$ and $(G^j)_{j=0}^\infty$ be positive sequences. Assume that $y^0 = 0$ and the discrete Caputo derivative $\partial_h^\alpha$ is constructed as the convolution quadrature (\ref{eqn:CQCaputo}). Then, the inequality
	\begin{equation}
		\label{eqn:GronwallAssumption}
		\partial_h^\alpha y^n \leq G^n,
	\end{equation}
	implies that
	\begin{equation}\label{eqn:GronwallInduction}
		y^n \leq \sum_{j=1}^n b_{n-j} G^j,
	\end{equation}
	where $b_j$ is defined as the CQ weight for the fractional integral \eqref{eqn:CQFracInt}.
\end{lem}
\begin{proof}
	The proof proceeds by mathematical induction. For $n=1$ from (\ref{eqn:CQCaputo}) we have that $w_0 > 0$ and
	\begin{equation}
		y^1 \leq \frac{1}{w_{0}} G^1,
	\end{equation}
	but $w_{0} = -w_{1} = h^{-\alpha} \delta(0)^{-\alpha} = b_0^{-1} $ and, hence, $y^1 \leq b_0 G^1$. This proves the initial step. We then assume that (\ref{eqn:GronwallInduction}) holds for $y^j$ with $j=1,...,n-1$. Then, by the definition of the quadrature (\ref{eqn:CQCaputo}) and our assumption (\ref{eqn:GronwallAssumption}) we have
	\begin{equation}
		y^n  \leq \frac{1}{w_{0}}\left(- \sum_{j=1}^{n-1} w_{n-j} y^j + G^n\right).
	\end{equation}
	Since $w_{j} < 0$ for $j \geq 1$, the inductive assumption then gives 
	\begin{equation}
		y^n  \leq \frac{1}{w_{0}}\left(- \sum_{j=1}^{n-1} \sum_{k=1}^j w_{n-j} b_{j-k} G^k + G^n\right),
	\end{equation}
	or, by changing the order of summation,
	\begin{equation}
		y^n  \leq \frac{1}{w_{0}}\left( - \sum_{k=1}^{n-1} \left(\sum_{j=k}^{n-1} w_{n-j} b_{j-k} \right) G^k + G^n\right),
	\end{equation}
	and we can focus on the resulting double sum. If we change the variable to $i = n-j$ we can write
	\begin{equation}
		\sum_{j=k}^{n-1} w_{n-j} b_{j-k} = \sum_{i=1}^{n-k} w_{i} b_{n-k-i}.
	\end{equation}
	The series above can be computed using generating functions. To see this, consider the following Cauchy product of power series
	\begin{equation}
		1 = \left(\frac{\delta(\zeta)}{h}\right)^{\alpha}\left(\frac{\delta(\zeta)}{h}\right)^{-\alpha} = \left(\sum_{m=0}^\infty w_{m}\zeta^m\right)\left(\sum_{m=0}^\infty b_{m} \zeta^m\right) = \sum_{m=0}^\infty \left(\sum_{i=0}^{m} w_{i} b_{m-i}\right) \zeta^m.
	\end{equation}
	Therefore, when $m>0$ the coefficients of the rightmost power series vanish, leading to
	\begin{equation}
		-\sum_{i=1}^{n-k} w_{i} b_{n-k-i} = w_{0}b_{n-k},
	\end{equation}
	hence,
	\begin{equation}
		y^n \leq \sum_{k=1}^{n-1} b_{n-k} G^k + \frac{1}{w_{0}}G^n.
	\end{equation}
	An observation that $w_{0} = b_0^{-1}$ finishes the inductive step, and we have proved (\ref{eqn:GronwallInduction}). 
\end{proof}

We can now proceed to the Gr\"onwall inequality for the discrete Caputo derivative. The idea of utilizing the classical result for the Gr\"onwall inequality for the fractional integral stated in Proposition \ref{prop:GronwallInt} was, to our knowledge, first used in the author's previous paper \cite{del2023numerical}. This approach has the advantage of allowing for the singular behavior of the bound for the fractional integral of the source.  

\begin{lem}[Discrete Gr\"onwall inequality for the convolution quadrature]
	\label{lem:Gronwall}
	Let $(y^j)_{j=0}^\infty$ and $(F^j)_{j=0}^\infty$ be positive sequences of numbers such that there exist $A(h)$, $B(h)>0$ bounding the discrete fractional integral in the following way
	\begin{equation}
		\label{eqn:GronwallF}
		h^\alpha \sum_{j=0}^{n-1} (n-j)^{\alpha-1} F^{j+1} \leq A(h)+B(h) t_n^{\alpha-1}, \quad n\geq 1.
	\end{equation}
	Assume that $y^0 = 0$ and the discrete Caputo derivative $\partial_h^\alpha$ is constructed as the convolution quadrature (\ref{eqn:CQCaputo}) with weights that satisfy (\ref{eqn:CQCaputoWeights}). Then, the inequality
	\begin{equation}
		\label{eqn:GronwallAssumption2}
		\partial_h^\alpha y^n \leq \lambda_0 y^n + \lambda_1 y^{n-1} + F^n, \quad \lambda \geq 0,
	\end{equation}
	implies that there exists a constant $C(\alpha,\lambda_{0,1}, T)> 0$ and a time-step $h_0>0$ such that
	\begin{equation}
		y^n \leq C(\alpha, \lambda_{0,1}, T)(A(h)+B(h)t_n^{\alpha-1})
	\end{equation}
	for all $0<h\leq h_0 < 1$. For $\lambda_0 = 0$ the above is valid without any restriction on the time step $h$.
\end{lem}
\begin{proof}
	We start by using Lemma \ref{lem:CaputoIntegrated} to invert \eqref{eqn:GronwallAssumption2} and obtain
	\begin{equation}\label{eqn:GronwallInduction2}
		y^n \leq \sum_{j=1}^n b_{n-j}\left(\lambda_0 y^j + \lambda_1 y^{j-1} + F^j\right).
	\end{equation}
	Next, by recalling the fact known from the convolution quadrature theory that the weights $b_j$ approximate the continuous kernel, that is (see for example \cite{lubich2004convolution}, formula (2.6))
	\begin{equation}
		|b_0| \leq C h^{\alpha}, \quad |b_j| \leq C h (j h)^{\alpha-1} = C h^\alpha j^{\alpha-1}, \quad j\geq 1,
	\end{equation}
	for some constant $C=C(\alpha)$. Therefore, the above, definition of the weights $b_j$ by \eqref{eqn:CQFracInt}, and separating the last term in the sum (\ref{eqn:GronwallInduction2}) yields 
	\begin{equation}
		\begin{split}
			y^n &\leq C(\alpha) h^\alpha\left(\lambda_0 y^n + \lambda_0\sum_{j=1}^{n-1} (n-j)^{\alpha-1} y^j + \lambda_1 y^{n-1} + \lambda_1 \sum_{j=1}^{n-1} (n-j)^{\alpha-1} y^{j-1} \right.\\
			&\left.+ F^n + \sum_{j=1}^{n-1} (n-j)^{\alpha-1} F^j \right).
		\end{split}
	\end{equation}
	All terms involving $y^j$ for $j\leq n-1$ can be estimated by a common sum. To see this observe that in the sum of $y^{j-1}$ we can change the summation variable $j-1 \mapsto j$ and use the elementary inequality
	\begin{equation}\label{elemental_ineq}
		(n-j-1)^{\alpha-1} \leq 2^{1-\alpha} (n-j)^{\alpha-1}, \quad \mbox{for} \quad 0\leq j\leq n-2,
	\end{equation}
	to obtain
	\begin{equation}
		\begin{split}
			\lambda_0 &\sum_{j=1}^{n-1} (n-j)^{\alpha-1} y^j + \lambda_1 y^{n-1} + \lambda_1 \sum_{j=0}^{n-2} (n-j-1)^{\alpha-1} y^{j} \\
			&\leq \lambda_0 \sum_{j=1}^{n-1} (n-j)^{\alpha-1} y^j + \lambda_1 y^{n-1} + 2^{1-\alpha}\lambda_1 \sum_{j=0}^{n-2} (n-j)^{\alpha-1} y^{j} \leq (\lambda_0+2^{1-\alpha}\lambda_1)\sum_{j=0}^{n-1} (n-j)^{\alpha-1} y^j.
		\end{split}
	\end{equation}
	Next, fix any time-step $h_0>0$ for which $1-C(\alpha)\lambda_0 h_0^\alpha > 0$. Then, for $0<h \leq h_0$ we can factor out $y^n$,
	\begin{equation}
		y^n \leq \frac{C(\alpha) (\lambda_0+2^{1-\alpha}\lambda_1)}{1-C(\alpha) \lambda_0 h_0^\alpha} h^\alpha \sum_{j=0}^{n-1} (n-j)^{\alpha-1} y^j + \frac{C(\alpha)2^{1-\alpha}}{1-C(\alpha) \lambda_0 h_0^\alpha} h^\alpha \sum_{j=1}^{n} (n-j+1)^{\alpha-1} F^j,
	\end{equation}
	where we have used again \eqref{elemental_ineq},  with the index shifted by one, in the sum with $F^j$. This can be further estimated with the use of the assumption of bounded fractional integral of $F^n$, that is, using (\ref{eqn:GronwallF}) we obtain (after changing the summation variable $j \mapsto j+1$)
	\begin{equation}
		y^n \leq \frac{C(\alpha) (\lambda_0+2^{1-\alpha}\lambda_1)}{1-C(\alpha) \lambda_0 h_0^\alpha} \sum_{j=0}^{n-1} (n-j)^{\alpha-1} y^j + \frac{C(\alpha)2^{1-\alpha}}{1-C(\alpha) \lambda_0 h_0^\alpha} (A(h)+B(h)t_n^{\alpha-1}).
	\end{equation}
	Set $M = C(\alpha)/(1-C(\alpha) \lambda_0 h_0^\alpha)$. We can now invoke Proposition \ref{prop:GronwallInt} with $M_0 = B(h) M 2^{1-\alpha}$ , $M_1 = (\lambda_0+2^{1-\alpha}\lambda_1)M$, and $M_2 = A(h)M 2^{1-\alpha}$, to obtain
	\begin{equation}
		y^n \leq M\Gamma(\alpha) E_{\alpha,\alpha} (M_1 \Gamma(\alpha) t_n^\alpha) t_n^{\alpha-1} B(h) +M E_\alpha\left(M_1\Gamma(\alpha)t_n^\alpha\right) A(h).
	\end{equation}
	Finally, we have $t_n\leq T$, and after redefinition of the constant $C(\alpha, \lambda_{0,1}, T)$ we arrive at the conclusion. 
\end{proof}

\section{Fully discrete scheme for the quasilinear subdiffusion equation}
We can now proceed to the derivation of the fully discrete scheme to solve \eqref{eqn:MainPDE}. Take any test function $\chi \in H^1_0(\Omega)$, then by integrating by parts, we can obtain the following
\begin{equation}
	\label{eqn:MainPDEWeak}
	(\partial^\alpha_t u, \chi) + a(D(x,t,u); u, \chi) = (f(x,t,u), \chi), \quad \chi \in H^1_0(\Omega),
\end{equation}
with $u(0) = 0$. Here, we defined the form
\begin{equation}
	\label{eqn:aForm}
	a(w;u,v) := \int_\Omega w(x) \nabla u \nabla v dx, 
\end{equation}
where $w \in L^1(\Omega)$ and, if $w \geq 0$, $a$ is positive definite. To discretize the above in time, we use the convolution quadrature (\ref{eqn:CQCaputo}) for the Caputo derivative and the finite element method (FEM) for the spatial variables. Let $\mathcal{T}_h$ be the family of shape-regular quasi-uniform triangulations of $\overline{\Omega}$ with the maximal diameter $k = \max_{K\in\mathcal{T}_k} \text{diam}K$. By $V_k \subset H_0^1(\Omega)$ denote the standard continuous piecewise linear function space over $\mathcal{T}_k$ that vanish on the boundary $\partial\Omega$
\begin{equation}
	V_k := \left\{\chi_k \in C^0(\overline{\Omega}): \; \chi_k|_K \text{ is linear for all } K\in\mathcal{T}_k \text{ and } \chi_k|_{\partial\Omega} = 0 \right\}.
\end{equation} 
Therefore, denoting by $U^n \in V_k$ the numerical approximation of $u(t_n)$ we devise the semi-implicit scheme
\begin{equation}\label{eqn:NumericalScheme}
	\begin{cases}
		(\partial_h^\alpha U^n, \chi) + a(D(t_n, U^{n-1}); U^n, \chi) = (f(t_n, U^{n-1}), \chi), & \chi \in V_k, \; n \geq 1, \\
		U^0 = 0. & \\
	\end{cases}
\end{equation}

In what follows, we will utilize some notions of projecting a function on a finite-dimensional space. For example, we can use the orthogonal projection $P_k$ defined as
\begin{equation}
	(P_k u - u, \chi) = 0, \quad \chi \in V_k,
\end{equation} 
or the Ritz elliptic projection $R_k$ for fixed $0\leq t\leq T$
\begin{equation}
	\label{eqn:RitzProjection}
	a(D(t, u); R_k u(t) - u(t), \chi) = 0, \quad \chi \in V_k.
\end{equation}
The latter is particularly useful in the convergence proof. Observe that to find $R_k u(t)$ it is necessary to solve a \emph{linear} elliptic problem. From the general theory of PDEs we know the error estimates on these projections when $u\in C^1((0,T]; H_0^1(\Omega) \cap H^2(\Omega))$ \cite{Tho07,luskin1982smoothing}
\begin{equation}
	\label{eqn:ProjectionErrors}
	\begin{split}
		&\|u(t)-P_k u(t)\|_p \leq C k^{2-p} \|u(t)\|_2, \quad \|u(t)-R_k u(t)\|_p \leq C k^{2-p} \|u(t)\|_2, \\
		&\|(u(t)-R_k u(t))_t\|_p \leq C k^{2-p} \|u(t)\|_2.
	\end{split}
\end{equation}
Moreover, for $\|\nabla u\|_\infty < \infty$, we have 
\begin{equation}
	\label{eqn:RitzBounded}
	\|\nabla R_k u\|_\infty \leq C.
\end{equation}
Finally, note also that in all nonlinearities of the equation, that is, $D$ and $f$, the time step has been delayed by one in order to obtain a fully linear scheme for the solution to the nonlinear equation. Having the results from the previous section, it is straightforward to prove that the scheme is stable.

\begin{prop}[Stability]\label{prop:Stability}
	Let $U^n$ be the solution of (\ref{eqn:NumericalScheme}). Suppose that there exists a function $g=g(t)$ such that $\|f(t,u)\|\leq g(t)$ with $I^\alpha g(t) \leq F$. Then, we have
	\begin{equation}
		\|U^n\| \leq C(\alpha, T) F.
	\end{equation}
\end{prop}
\begin{proof}
	Let $\chi = U^n \in V_k$ in (\ref{eqn:NumericalScheme}), then from the Proposition \ref{prop:CQInner} and the Cauchy inequality \ref{prop:CQInner} we have
	\begin{equation}
		\|U^n\|(\partial_h^\alpha \|U^n\|) + a(D(t,U^{n}), U^n, U^n) \leq \|U^n\|g(t_n).
	\end{equation}
	Since, by definition (\ref{eqn:aForm}) the $a$-form is positive-definite we further have
	\begin{equation}
		\|U^n\|(\partial_h^\alpha \|U^n\|) \leq \|U^n\| g(t_n),
	\end{equation}
	or 
	\begin{equation}
		\partial_h^\alpha \|U^n\| \leq g(t_n).
	\end{equation}
	Now, notice that there exists a constant $C$ such that
	\begin{equation}
		h^\alpha \sum_{j=0}^{n-1} (n-j)^{\alpha-1} g(t_j) \leq C I^{\alpha} g(t_n) \leq C F,
	\end{equation}
	since the rectangle discretization of the fractional integral converges to the continuous one. The application of Lemma \ref{lem:Gronwall} ends the proof. 
\end{proof}

Before we proceed to our main result for the quasilinear equation we derive error estimates for the semilinear problem, with $D(x,t,u)=D(x,t)$.

\subsection{Error estimates for the semilinear problem}
For the semilinear problem with $D(x,t,u)=D(x,t)$ and $f(x,t,u)$, there are no error estimates available in the literature. In the linear case $f(x,t,u)=f(x,t)$ we do find a complete regularity analysis of the exact solution in \cite{JinLiZhou2019} and \cite{JinLiZhou2020}, together with error estimates of BDF1 and BDF2 based CQ time approximation schemes. In particular \cite[Theorem 4.2]{JinLiZhou2019} provides optimal pointwise error estimates for BDF1 based CQ scheme. Our result for the semilinear problem is the following. 

\begin{thm}\label{thm:Auxiliary}
	Let $u$ be the solution of \eqref{eqn:MainPDEWeak} with $D(x,t,u)=D(x,t)$ and $U^n$ its numerical approximation given by \eqref{eqn:NumericalScheme}. Suppose that there exist $A(h)$, $B(h)>0$ bounding the discrete fractional integral of the truncation error \eqref{eqn:CQCaputoTruncation} in the following way
	\begin{equation}
		\label{eqn:GronwallF}
		h^\alpha \sum_{j=0}^{n-1} (n-j)^{\alpha-1} \|\zeta_{j+1}(h)\| \leq A(h)+B(h) t_n^{\alpha-1}, \quad n\geq 1.
	\end{equation}
	Then, for sufficiently small $h>0$ the error satisfies
	\begin{equation}\label{eqn:AuxiliaryError0}
		\|U^n-u(t_n)\| \leq C (A(h) + B(h) t^{\alpha-1}_n + k^2).
	\end{equation} 
\end{thm}
\begin{proof}
	We start by decomposing the error in two terms
	\begin{equation}\label{eqn:ErrorDecomposition}
		U^n - u(t_n) = U^n - R_k u(t_n) + R_k u(t_n) - U(t_n) =: \theta^n + \rho^n.
	\end{equation}
	The decomposition of $e^n$ into $\rho^n$ and $\theta^n$ is very useful since the estimate on $\rho^n$ follows from general theory (\ref{eqn:ProjectionErrors}) while $\theta^n$ belongs to the finite-dimensional space $V_k$. Therefore, it is sufficient to find a bound on the latter error. Hence, observe that for any $\chi \in V_k$ from the definition of error decomposition into $\rho^n$ and $\theta^n$ we have
	\begin{equation}
		\begin{split}
			(\partial^\alpha_h \theta^n, \chi) + a(\widetilde{D}(t_n); \theta^n, \chi) 
			&= (\partial^\alpha_h U^n, \chi) - (\partial^\alpha_h R_k u(t_n), \chi) \\
			&+ a(\widetilde{D}(t_n); U^n, \chi) - a(\widetilde{D}(t_n); R_k u(t_n), \chi).
		\end{split}
	\end{equation}
	Now, we can use the equation \eqref{eqn:NumericalSchemeA} for $U^n$ and the definition of the Ritz projection (\ref{eqn:RitzProjection}) to obtain
	\begin{equation}
		\begin{split}
			(\partial^\alpha_h \theta^n, \chi) + a(\widetilde{D}(t_n); \theta^n, \chi) 
			&= (g(t_n, U^n), \chi) - (\partial^\alpha_h R_k u(t_n), \chi) \\
			& - a(\widetilde{D}(t_n); u(t_n), \chi).
		\end{split}
	\end{equation}
	Next, we use the differential equation for $u$ to express the last term
	\begin{equation}
		\begin{split}
			(\partial^\alpha_h \theta^n, \chi) + a(\widetilde{D}(t_n); \theta^n, \chi) 
			&= (g(t_n, U^n), \chi) - (\partial^\alpha_h R_k u(t_n), \chi) \\
			&+ (\partial_t^\alpha u(t_n), \chi) -(g(t_n, u(t_n)), \chi).
		\end{split}
	\end{equation}
	Rearranging terms further gives
	\begin{equation}
		(\partial^\alpha_h \theta^n, \chi) + a(\widetilde{D}(t_n); \theta^n, \chi) = ((\partial_t^\alpha -\partial^\alpha_h) u(t_n), \chi) - (\partial^\alpha_h \rho^n, \chi) + (g(t_n, U^n) - g(t_n, u(t_n)), \chi).
	\end{equation}
	We can take $\chi = \theta^n \in V_k$ and use the assumption that $D \geq D_- > 0$ and that $f$ is Lipschitz to arrive at
	\begin{equation}
		(\partial^\alpha_h \theta^n, \theta^n) + D_- \|\nabla \theta^n\|^2 \leq \left(\|(\partial_t^\alpha -\partial^\alpha_h)u(t_n)\| + \|\partial^\alpha_h \rho^n\| + \|U^n - u(t_n)\|\right) \|\theta^n\|,
	\end{equation}
	where we have used the Cauchy-Schwarz inequality. The gradient term can now be bounded from below by $0$, the error $\|U^n - u(t_n)\|$ can be again decomposed according to \eqref{eqn:ErrorDecomposition}, and we can invoke Proposition \ref{prop:CQInner} to obtain
	\begin{equation}\label{eqn:ErrorInequality0}
		\|\theta^n\|\partial^\alpha_h \|\theta^n \| \leq \left(\|(\partial_t^\alpha -\partial^\alpha_h)u(t_n)\| + \|\partial^\alpha_h \rho^n\| + \|\rho^n\| + \|\theta^n\|\right) \|\theta^n\|.
	\end{equation}
	Canceling the $\|\theta^n\|$ factor we can see that the error evolution decomposes into several terms: truncation error of the CQ scheme and the Ritz projection errors. By \eqref{eqn:CQCaputoTruncation} the first term satisfies 
	\begin{equation}
		\|(\partial_t^\alpha -\partial^\alpha_h)u(t_n)\| \leq \|\zeta_n(h) \|,
	\end{equation}
	while denoting $\rho(t)=u(t)-R_ku(t)$, so that $\rho^n=\rho(t_n)$, the second can be estimated as follows
	\begin{equation}
		\begin{split}
			\|\partial_h^\alpha \rho^n\| 
			&\leq \|(\partial_h^\alpha - \partial^\alpha_t) \rho(t_n)\| + \|\partial_t^\alpha \rho(t_n)\| \leq \|\zeta_n(h)\| + \frac{1}{\Gamma(1-\alpha)} \int_0^{t_n} (t_n-s)^{-\alpha} \|(u - R_k u)_t(s)\| ds \\
			&\leq \|\zeta_n(h)\| + C(u)\frac{k^2}{\Gamma(1-\alpha)} \int_0^{t_n} (t_n-s)^{-\alpha} ds = \|\zeta_n(h)\| + C(u) k^2 \frac{t_n^{1-\alpha}}{\Gamma(2-\alpha)} \\
			&\leq \|\zeta_n(h)\| + C(\alpha, T, u) k^2,
		\end{split}
	\end{equation}
	where we used the projection error estimates \eqref{eqn:ProjectionErrors}. An estimate for $\|\rho^n\|$ immediately follows from the same formula. Hence, from \eqref{eqn:ErrorInequality0} we have
	\begin{equation}
		\partial^\alpha_h \|\theta^n \| \leq C \left(\|\zeta_n(h)\| + k^2 + \|\theta^n\|\right).
	\end{equation}
	The above is ready to be used in Gr\"onwall inequality (Lemma \ref{lem:Gronwall}) with $\lambda_0 = C$, $\lambda_1 = 0$, and $F^n = C(\|\zeta_n(h)\| + k^2)$. To this end we have to find the bound on the discrete fractional integral of $F^n$. We have
	\begin{equation}
		h^\alpha \sum_{j=0}^{n-1} (n-j)^{\alpha-1} F^{j+1} = h^\alpha \sum_{j=0}^{n-1} (n-j)^{\alpha-1} \|\zeta_j(h)\| + k^2 h^\alpha \sum_{j=0}^{n-1} (n-j)^{\alpha-1} 
	\end{equation}
	The first sum above can be estimated from the assumption \eqref{eqn:GronwallF} while the other by noticing that we can transform it into a Riemann sum of the corresponding integral
	\begin{equation}
		k^2 h^\alpha \sum_{j=0}^{n-1} (n-j)^{\alpha-1} \leq k^2 h^\alpha n^{\alpha} \frac{1}{n} \sum_{j=0}^{n-1} \left(1-\frac{j}{n}\right)^{\alpha-1}
		\leq t_n^\alpha k^2 \int_0^1 (1-x)^{\alpha-1} dx \leq C k^2. 
	\end{equation}
	Finally, note that by the initial decomposition of the error into $\rho^n$ and $\theta^n$ we have
	\begin{equation}
		\|U^n - u(t_n)\| \leq \|\rho^n\| + \|\theta^n\| \leq C (A(h) + B(h) t_n^{\alpha-1} + k^2),
	\end{equation}
	where we again used the Ritz projection error estimate \eqref{eqn:ProjectionErrors} and by that concluded the proof.
\end{proof}

For BDF$(p)$ based CQ we can compute the error explicitly.
\begin{lem}\label{lem:ABError}
	If the Caputo derivative is discretized with the BDF$(p)$-CQ, then $A(h) = 0$ and
	\begin{equation}\label{eqn:Bh}
		B(h) = C\begin{cases}
			h \ln (h^{-1}), & p = 1, \\
			h, & p \geq 2, \\
		\end{cases}
	\end{equation}
	where $C>0$ is a $h$-independent constant. 
\end{lem}
\begin{proof}
	To prove the assertion, we have to find out the form of $A(h)$ and $B(h)$ defined in Theorem \ref{thm:Auxiliary}. They come from the bound of the discrete fractional integral of $\zeta_n(h)$, that is, the truncation error of the Caputo derivative as in \eqref{eqn:CQCaputoTruncation}. First, assume that $p=1$, that is, we consider the Euler scheme. From \eqref{eqn:CQCaputoTruncationEstimate} we have
	\begin{equation}
		h^\alpha \sum_{j=0}^{n-1} (n-j)^{\alpha-1} \|\zeta_{j+1}(h)\| \leq C h^\alpha \sum_{j=0}^{n-1} (n-j)^{\alpha-1} (j+1)^{-1} \leq C h^\alpha n^{\alpha-1} \frac{1}{n} \sum_{j=1}^{n} \left(1+\frac{1}{n}-\frac{j}{n}\right)^{\alpha-1} \left(\frac{j}{n}\right)^{-1},
	\end{equation}
	where we have changed the summation order via $j \mapsto j -1$. As can be seen, the resulting expression is the Riemann sum of a convergent integral, hence with a suitable choice of the constant
	\begin{equation}
		h^\alpha \sum_{j=0}^{n-1} (n-j)^{\alpha-1} \|\zeta_{j+1}(h)\| \leq C h^\alpha n^{\alpha-1} \int_{\frac{1}{n}}^1 (1-x)^{\alpha-1} \frac{dx}{x}. 
	\end{equation}
	This integral can be evaluated exactly with the help of the hypergeometric function, but for our needs, we only have to find its leading order behavior for large $n$. To this end, for arbitrary $\frac 1n < \epsilon  < 1$, we split the integral into two parts
	\begin{equation}
		\begin{split}
			\int_{\frac{1}{n}}^1 (1-x)^{\alpha-1} \frac{dx}{x} &= \int_{\frac{1}{n}}^\epsilon (1-x)^{\alpha-1} \frac{dx}{x} + \int_{\epsilon}^1 (1-x)^{\alpha-1} \frac{dx}{x} \leq (1-\epsilon)^{\alpha-1} \int_{\frac{1}{n}}^\epsilon \frac{dx}{x} + \frac{1}{\epsilon} \int_{\epsilon}^1 (1-x)^{\alpha-1} dx \\
			&= (1-\epsilon)^{\alpha-1} \ln \left(n \epsilon\right) + \frac{(1-\epsilon)^\alpha}{\epsilon\alpha} \leq C \ln n,
		\end{split}
	\end{equation}
	and, hence, for sufficiently small $h>0$
	\begin{equation}
		\begin{split}
			h^\alpha \sum_{j=0}^{n-1} (n-j)^{\alpha-1} \|\zeta_{j+1}(h)\| &\leq C (nh)^{\alpha-1} h \ln n = C t_n^{\alpha-1} h (\ln (nh) - \ln h) \\
			&\leq C t_n^{\alpha-1} h (\ln T  - \ln h ) \leq C t_n^{\alpha-1} h \ln h^{-1}.
		\end{split}
	\end{equation}
	This gives us $A(h) = 0$, $B(h) = C h \ln h^{-1}$. Noticing that $t_{n}^{\alpha-1}h \leq h^\alpha$ proves the case with $p=1$. 
	
	Now, assume that $p \ge 2 $. By a similar reasoning we can identify the Riemann sum of the corresponding integral and obtain
	\begin{equation}
		\begin{split}
			h^\alpha \sum_{j=0}^{n-1} (n-j)^{\alpha-1} \|\zeta_{j+1}(h)\| &\leq C h^{1+\alpha} \int_0^{t_n} (t_n-s)^{\alpha-1} (s+h)^{-\alpha-1} ds \\
			&= C (nh)^{\alpha-1}\frac{n h}{\alpha(1+n)} \leq \frac{1}{\alpha}C t_n^{\alpha-1} h
		\end{split}
	\end{equation}
	since the appearing integral has an exact primitive $-\frac{1}{\alpha(t_n+h)}\left(\frac{t_n-s}{h+s}\right)^\alpha$. This time we have $A(h) = 0$ and $B(h) = C h$. Combining the case $p\geq 2$ with \eqref{eqn:ConvergenceError} finishes the proof. 
\end{proof}

\begin{rem}
	Taking into account Lemma \ref{lem:ABError} we observe that the error bound in Theorem \eqref{thm:Auxiliary} is pointwise \emph{optimal} for $p=1$ (apart from the slowly increasing logarithmic term ) and consistent with the results in \cite{JinLiZhou2019} for the linear problem with a time-dependent diffusion coefficient. For $p=2$, we only prove convergence with order 1, whereas in \cite{JinLiZhou2020} an error estimate of the second order is obtained for the linear problem by using a correction of the original BDF(2)-CQ scheme. In this paper we do not consider the introduction of corrections terms and work directly with the original truncation error \eqref{eqn:CQCaputoTruncation}.
\end{rem}

\subsection{Error estimates for the quasilinear problem}
As will be shown below, the main tool in proving our error estimates for the quasilinear problem is to use an intermediate linear problem
\begin{equation}
	\label{eqn:AuxPDElin}
	(\partial^\alpha_t v, \chi) + a(\widetilde{D}(t); v, \chi) = (\widetilde{f}(t), \chi), \quad \chi \in H^1_0(\Omega),
\end{equation}
with $v(0)= 0$, diffusivity $\widetilde{D}=\widetilde{D}(\cdot,t,u(t))$ and source $\widetilde{f}(t)=f(\cdot, t, u(t))$, which satisfy analogous regularity assumptions as $D$ and $f$, that is are $C^2(\Omega, \mathbb{R}_+,\mathbb{R})$ and the positive diffusivity is separated from zero. We discretize the above with the following implicit scheme
\begin{equation}\label{eqn:NumericalSchemeA}
	\begin{cases}
		(\partial_h^\alpha V^n, \chi) + a(\widetilde{D}(t_{n}); V^n, \chi) = (\widetilde{f}(t_n), \chi), & \chi \in V_k, \; n \geq 1, \\
		V^0 = 0. & \\
	\end{cases}
\end{equation}
The above will play a major tool in finding globally optimal error estimates for our main scheme \eqref{eqn:NumericalScheme}. In order to derive optimal stability and error estimates for the quasilinear problem we need stability estimates in the $H^1$ norm for the auxiliary problem \eqref{eqn:AuxPDElin}.

The next result collects stability and pointwise error bounds for the time discretization by CQ based on BDF1 and BDF2 of a linear problem \eqref{eqn:MainPDE} with a constant diffusion coefficient in stronger norms than the one of $L^2(\Omega)$. More precisely, following \cite{cuesta2006convolution}, we consider the following equation 
\begin{equation}\label{eqn:linearconstantA}
	\partial_t^{\alpha} w(t) + A w (t) = g(t), \quad w(0)=0,
\end{equation} 
on a Banach space $X$, which can be either infinite dimensional, such as $X=L^2(\Omega)$, or finite dimensional, such as the Galerkin space $X=V_k$. We assume $-A$ is a sectorial operator and denote, for $\rho>0$, $X_{\rho}=D(A^{\rho})$, with norm $\|w\|_{\rho}=\|A^{\rho} w\|$. This is in particular the case if $A=-\Delta$, with $D(A)=H^1_0(\Omega) \cap H^2(\Omega)$, $D(A^{\rho}) = \dot{H}^{2\rho}(\Omega)$ and, in particular, $D(A^{\frac 12}) = \dot{H}^{1}(\Omega) = H^1_0(\Omega)$. Typically the discretization of a second order elliptic operator $A$ by using finite differences of finite elements preserves the sectorial character and the resulting discrete operator $-A_k$ satisfies resolvent bounds which are independent of the spatial discretization parameters, \cite{Ashyralyev2012,Bakaev2003}. The spaces $D(A^{\rho})$ can be defined via the spectral decomposition of the operator $A$, see \cite{JinLiZhou2020}. In what follows we need a bound for the gradient of the numerical solution $V^n$ to , and thus we derive stability and error estimates in the $\|\cdot\|_{1/2}$ norm, according to this notation.

\begin{lem}\label{lem:errh1linconstant}
	Let $w$ be the solution to \eqref{eqn:linearconstantA} and $W_n$ its numerical approximation at time $t_n$ defined by
	\begin{equation}\label{eqn:CQlinearconstantA}	
		\partial_h^\alpha W^n + A W^n = g(t_n); \quad W^0=0, \quad  n \geq 1,
	\end{equation}
	where the Caputo derivative is discretized with the BDF$(p)$-CQ.  Then 
	%\red{
		%\begin{equation}\label{boundgrad}
		%\|W^n\|_{1/2} \le C \max_{0\le s\le t_n} \|g(s)\|
		%\end{equation}}
		%and
		\begin{equation}\label{errorH1_lin_p1}
			\|w(t_n)-W^n\|_{1/2} \le C  h t_n^{\frac{\alpha}{2} -1} \left( \|g(0)\| + t \max_{0\le s\le t} \|g_k'(s)\| \right).
		\end{equation}
		%	If $g(0)=0$ and the scheme is based on BDF2-CQ, then
		%	\begin{equation}\label{errorH1_lin_p1}
			%		\|v(t_n)-W^n \|_{1/2} \le C  h^2 t_n^{\frac{\alpha}{2} -1} \left( \|g_k'(0)\| + t \max_{0\le s\le t} \|g_k''(s)\| \right).
			%	\end{equation}
		%\color{teal}I have removed the special case $g(0) = 0$ since I think that this is very special and in that case all $BDF$s %give the same bound. But please reinstate/correct it if you wish. 
	\end{lem}
	\begin{proof} 
		By applying Laplace transform techniques $w$ solves
		\[
		w(t) + \partial_t^{-\alpha} A(t) w(t) = \partial_t^{-\alpha} g(t),
		\]
		or equivalently 
		\[
		w(t) = (I + \partial_t^{-\alpha} A)^{-1} \partial_t^{-\alpha} g(t).
		\]
		By the preservation of the composition rule by  CQ schemes, the discretization of \eqref{eqn:linearconstantA}  is equivalent to the scheme
		\[
		w_h(t) = (I + \partial_h^{-\alpha} A)^{-1} \partial_h^{-\alpha} g(t),
		\]
		where $w_h(t)$ denotes the continuous extension in time of the CQ scheme and satisfies $w_h(t_n)=W^n$.
		The error is then
		\[
		w(t)-w_h(t) = \left[(I + \partial_t^{-\alpha} A)^{-1} \partial_t^{-\alpha} -(I + \partial_h^{-\alpha} A)^{-1} \partial_h^{-\alpha} \right] g(t),
		\]
		so that
		\[
		A^{1/2} (w(t)-w_h(t)) = \left[A^{1/2} (I + \partial_t^{-\alpha} A)^{-1} \partial_t^{-\alpha} -A^{1/2} (I + \partial_h^{-\alpha} A)^{-1} \partial_h^{-\alpha} \right] g(t)
		\]
		and $A^{1/2}w_h$ is the CQ approximation of a convolution whose kernel has Laplace transform 
		\[
		G(s)=A^{1/2} (I + s^{-\alpha} A)^{-1} s^{-\alpha}.
		\]
		The resolvent estimates in \cite[Lemma 2.2]{cuesta2006convolution} imply that
		\[
		\|G(s)\|_{L^2\to L^2} \le M|s|^{-\alpha/2}.
		\]
		%The stability result in \cite[Equation (5.4)]{Lu88I} gives the following bound for the CQ weights associated to $G$, call %them $\widetilde{\omega}_n$,
		%\[
		%\|\widetilde{\omega}_0\| \le C h^{\frac{\alpha}{2}}, \qquad \|\widetilde{\omega}_n\| \le C h^{\frac{\alpha}{2}} %n^{\frac{\alpha}{2}-1}, \quad n\ge 1,
		%\]
		%following the bound
		%\[
		%\|A^{1/2}W^n\| \le C h^{\frac{\alpha}{2}} + C  h^{\frac{\alpha}{2}} \sum_{j=0}^{n-1} (n-j)^{\frac{\alpha}{2}-1} \|g(t_j)\|,
		%\]
		%which is the discrete fractional integral of $g$ of order $\alpha/2$ and is bounded under our assumptions.
		%}
	By \cite[Theorem 3.1]{Lu88I} it follows the error bound 
	\[
	\|w(t)-w_h(t)\|_{1/2} \le C t^{\frac{\alpha}{2} -1} \left[ h\|g(0)\| + \dots +h^{p-1}\|g^{(p-2)}(0)\| + h^p\left( \|g^{(p-1)}(0)\| + t \max_{0\le s\le t} \|g^{(p)}(s)\| \right)   \right].
	\]
	For $p=1$, the error estimate at $t_n$ follows, with $W^n=V_h(t_n)$, 
	\[
	\|w(t_n)-W^n\|_{1/2} \leq C h t_n^{\frac{\alpha}{2} -1} \left( \|g(0)\| + t_n \max_{0\le s\le t_n} \|g'(s)\| \right),
	\]
	which dominates the overall behavior for all $p\geq 1$ unless compatibility conditions $g^{(i)}(0) = 0$ with $i=0,...,p-1$ are satisfied. The proof is complete.
	%A similar estimate follows for $p=2$ if $g(0)\ne 0$. If $g(0)=0$, then the following second-order error bound holds
	%\[
	%\|w(t_n)-W^n\|_{1/2} \le C h^2 t_n^{\frac{\alpha}{2} -1} \left( \|g'(0)\| + t_n \max_{0\le s\le t_n} \|g''(s)\| \right).
	%\]
\end{proof}

Next, we address the optimal error bounds in time and space for the discretization of the linear problem with time-dependent coefficients \eqref{eqn:AuxPDElin}. 
\begin{lem}\label{lem:errh1aux}
	Let $v$ be the solution to \eqref{eqn:AuxPDElin} and $V_n$ its numerical approximation at time $t_n$ according to \eqref{eqn:NumericalSchemeA}, that is
	\begin{equation}\label{eqn:NumericalSchemeAlin}
		\begin{cases}
			(\partial_h^\alpha V^n, \chi) + a(\widetilde{D}(t_{n}); V^n, \chi) = (\widetilde{f}(t_n), \chi), & \chi \in V_k, \; n \geq 1, \\
			V^0 = 0, & \\
		\end{cases}
	\end{equation}
	where the Caputo derivative is discretized with the BDF$(p)$-CQ. Then, 
	%\red{
		%\begin{equation}\label{boundgradV}
		%\|\nabla V^n \| \le C \|g\|_{L^{\infty}(0,t_n; L^2({\Omega}))} 
		%\end{equation}
		%}
	%and
	\begin{equation}\label{errorH1_lin_p1}
		\|\nabla\left(v(t_n)-V^n\right)\| \le C \left( h t_n^{\frac{\alpha}{2} -1} \left( \|g(0)\|+ t_n \max_{\ell=0,1,2} \|g^{(\ell)}\|_{L^{\infty}(0,t_n; L^2({\Omega}))} \right) + k \right).
	\end{equation}
	%If $g(0)=0$ and the scheme is based on BDF2-CQ, then
	% \begin{equation}\label{errorH1_lin_p1}
		%\|\nabla\left(v(t_n)-V^n\right)\| \le C \left( h^2 t_n^{\frac{\alpha}{2} -1} \left( \dots \right) + k l_k^2 %\|g\|_{L^{\infty}(0,t_n; L^2({\Omega}))} \right).
		%   \end{equation}
\end{lem}
\begin{proof}
	We decompose the above error as follows
	\[
	\|\nabla v(t_n) -\nabla V^n \| \leq \| \nabla v(t^n) - \nabla v_k(t_n)\|+ \|\nabla v_k(t_n) - \nabla V^n\|,
	\]
	for $v_k(t)$ the solution to the semidiscrete problem in space (a system of FDEs in $V_k$)
	\begin{equation}\label{semidiscrete}
		\partial_t^\alpha v_k(t) + A_k(t) v_k(t) = P_k g(t);  \quad v_k(0) = 0, \\
	\end{equation}
	where $A_k(t) : V_k \to V_k$ is the operator defined by 
	\[
	\left(A_k(t) v_k, \chi \right) = \left(a(\tilde{D}(t)); v_k, \chi \right), \quad \forall v_k, \chi \in V_k.
	\]
	By repeating exactly the same energy arguments as in the proof of Theorem \eqref{thm:Auxiliary} but for the time-continuous case, or by taking $h\rightarrow 0^+$ in \eqref{eqn:AuxiliaryError0}, we can show that
	\begin{equation}
		\| v(t_n) - v_k(t_n) \| \leq C k^2.
	\end{equation}
	Now, using the inverse inequality (\cite{Tho07}, formula (1.12))
	\begin{equation}
		\|\nabla \chi\| \leq C k^{-1} \|\chi\|, \quad \chi \in V_k,
	\end{equation}	
	we can estimate
	\[
	\| \nabla v(t_n)-\nabla V^n\| \leq k^{-1} \| v(t_n) - v_k(t_n) \|+ \|\nabla v_k(t_n) - \nabla V^n\| \le  C k + \| \nabla v_k(t_n)-\nabla V^n \|.
	\]
	To derive \eqref{errorH1_lin_p1}, we observe that $v_k$ solves the linear problem with time dependent diffusion operator
	\begin{equation}\label{linearconstantA}
		\partial_t^{\alpha} v_k(t) + A_k(t) v_k(t) = P_k g(t), \quad v_k(0)=0,
	\end{equation}             
	in the (finite dimensional) space $V_k$. By applying the same technique as in \cite{JinLiZhou2020}, we write this problem as 
	\begin{equation}\label{linearconstantA}
		\partial_t^{\alpha} v_k(t) + A_k(0) v_k(t) = P_k g(t) + (A_k(0)-A_k(t))v_k, \quad v_k(0)=0.
	\end{equation}   
	We set $F_k(t)=P_k g(t) + (A_k(0)-A_k(t))v_k$ and consider the auxiliary problem 
	\begin{equation}\label{aux2}
		\partial_t^{\alpha} w(t) + A_k(0)w(t) = F_k(t), \quad w(0)=0,
	\end{equation}   
	whose solution satisfies $w=v_k$. From the proof of \cite[Lemma 9]{JinLiZhou2020} and \cite[Theorem 2]{JinLiZhou2020}, we have
	\[
	\|F_k'(t)\|_{L^2} \le (a_0+a_1 t)\|P_kg(0)\| + (b_0 +b_1t)t \max_{0\le s\le t}\|P_k g'(s)\| 
	+ c_0 t^4 \max_{0\le s\le t}\|P_kg''(s)\|.
	\]
	Thus, by Lemma~\ref{lem:errh1linconstant}, we can estimate 
	\[
	\|w(t_n)-W^n\|_{1/2} \leq C h t_n^{\frac{\alpha}{2} -1} \left( \|P_kg(0)\| + t_n \max_{0\le s\le t_n} \|F_k'(s)\| \right),
	\]
	and the result follows.		
\end{proof}

\begin{cor}\label{cor:GradientBound}
	Let the assumptions of Theorem \ref{thm:Auxiliary} for equation \eqref{eqn:AuxPDElin}. If additionally, $\|\nabla v(t)\|_\infty < \infty$ and $h k^{-1} \leq t_n^{1-\alpha/2}$, then we have 
	\begin{equation}
		\|\nabla V^n\|_\infty \leq C.
	\end{equation}
\end{cor}
\begin{proof}
	By the inverse inequality (\cite{Tho07}, formula (13.11))
	\begin{equation}\label{eqn:InverseInequalities}
		\|\nabla \chi\|_\infty \leq C k^{-1} \|\nabla \chi\|, \quad \chi \in V_k,
	\end{equation}
	we can estimate
	\begin{equation}\label{eqn:MaxNormEst}
		\|\nabla V^n\|_\infty \leq \|\nabla v(t_n)\|_\infty + \|\nabla V^n - \nabla v(t_n)\|_\infty \leq \|\nabla v(t_n)\|_\infty + C k^{-1} \|\nabla V^n - \nabla v(t_n)\|.
	\end{equation}
	Now, due to Lemma \ref{lem:errh1aux} we have
	\begin{equation}
		\|\nabla V^n - \nabla v(t_n)\| \leq C \left(h t_n^{\frac{\alpha}{2}-1} + k\right).
	\end{equation}
	Therefore, by the assumption $\|\nabla v\|_\infty < \infty$, we have
	\begin{equation}
		\|\nabla V^n\|_\infty \leq C \left(1 + h k^{-1} t_n^{\frac{\alpha}{2}-1} + 1\right),
	\end{equation}
	which by the restriction on the discretization, that is $h k^{-1} \leq t_n^{1-\frac{\alpha}{2}}$, concludes the proof.

\end{proof}

\begin{rem}
	In the case of the one-dimensional domain $\Omega \subset \mathbb{R}$, the restriction of the time step can be relaxed. That is, the inverse inequality \cite[(4.5.11) Theorem]{brenner2008mathematical}
	\begin{equation}
		\|\nabla \chi\|_\infty \leq C k^{-\frac{1}{2}} \|\nabla \chi\|, \quad \chi \in V_k,
	\end{equation}
	introduces in \eqref{eqn:MaxNormEst} the prefactor $k^{-1/2}$. Hence, we only require \begin{equation} \label{cfld1}
		h k^{-1/2} \leq t_n^{1-\alpha/2}.
	\end{equation}
\end{rem}

Now, we have all ingredients to proceed to our main result. 

\begin{thm}\label{thm:Convergence}
	Let $u$ is the solution of \eqref{eqn:MainPDEWeak} of regularity \eqref{eqn:RegularityAssumtion} and $U^n$ its numerical approximation given by \eqref{eqn:NumericalScheme} with the Caputo derivative discretized with the BDF$(p)$-CQ. If the following condition is satisfied, 
	\begin{equation}\label{cfl}
		h k^{-1} \leq t_n^{1-\alpha/2},
	\end{equation}
	then, for sufficiently small $h>0$, the error of the scheme satisfies
	\begin{equation}\label{eqn:ConvergenceError}
		\|U^n-u(t_n)\| \leq C
		\begin{cases}
			B(h)\max\{t_n^{\alpha-1}, t_n^\frac{\alpha-1}{2}\} +k^2 , & \alpha > \frac{1}{2} \\
			B(h)(t_n^{\alpha-1}+(\ln h^{-1})^{1/2} t_n^{\frac{\alpha-1}{2}}) + k^2, & \alpha = \frac{1}{2} \\
			B(h)(t_n^{\alpha-1}+h^{\alpha-1/2} t_n^{\frac{\alpha-1}{2}}) + k^2, & \alpha < \frac{1}{2} \\
		\end{cases},
	\end{equation}
	for some $C=C(\alpha, T, u)>0$ with
	\begin{equation}
		B(h) = 
		\begin{cases}
			h \ln(h^{-1}), & p = 1, \\
			h, & p \geq 2. \\
		\end{cases}
	\end{equation}
\end{thm}
\begin{proof}
	The idea is to make the following decomposition
	\begin{equation}\label{eqn:ErrorDecomposition}
		\|U^n - u(t_n)\| \leq \|U^n - V^n\| + \|V^n - u(t_n)\| =: \|e^n\| + \|\eta^n\|,
	\end{equation}
	where $V^n$ is the numerical approximation of $u$ defined in \eqref{eqn:NumericalSchemeA} with the diffusivity and the source chosen according to 
	\begin{equation}\label{eqn:DiffusivitySource}
		\widetilde{D}(t)=D(t,u(t)), \qquad \widetilde{f}(t)=f(t,u(t)). 
	\end{equation}
	Note that the bound for $\|\eta^n\|$ is given by Theorem~\ref{thm:Auxiliary} and Lemma \ref{lem:ABError} as
	\begin{equation}\label{eqn:EtaBound}
		\|\eta^n\| \leq C (B(h) t_n^{\alpha-1} + k^2),
	\end{equation} 
	with $B(h)$ given by \eqref{eqn:Bh}. We will find an estimate for $e^n$ by writing an appropriate error equation. By subtracting \eqref{eqn:NumericalSchemeA} from \eqref{eqn:NumericalScheme} we obtain
	\begin{equation}
		\begin{split}
			(\partial_h^\alpha e^n, \chi) &+ a(D(t_n, U^{n-1}); e^n, \chi) + a(D(t_n, U^{n-1})-D(t_n, u(t_n)); V^n, \chi) \\
			&= (f(t_n, U^{n-1}) - f(t_n, u(t_n)), \chi).
		\end{split}
	\end{equation}
	Hence, putting $\chi = e^n \in V_k$ and using the bound for $D$ along with positive-definiteness of $a$ we have
	\begin{equation}
		(\partial_h^\alpha e^n, \chi) + D_- \|\nabla e^n\|^2 \leq - a(D(t_n, U^{n-1})-D(t_n, u(t_n)); V^n, e^n) + (f(t_n, U^{n-1}) - f(t_n, u(t_n)), e^n). 
	\end{equation}
	Using the definition of the $a$-form \eqref{eqn:aForm} and the Lipschitz regularity of the diffusivity \eqref{eqn:Assumptions} we can estimate as follows using the Cauchy-Schwarz inequality
	\begin{equation}
		a(D(t_n, U^{n-1})-D(t_n, u(t_n)); V^n, e^n) \leq \|\nabla V^n\|_\infty \| U^{n-1} - u(t_n) \| \|\nabla e^n\|.
	\end{equation}
	Furthermore, by the Lipschitz assumption \eqref{eqn:Assumptions} on $f$,
	\begin{equation}
		(\partial_h^\alpha e^n, e^n) + D_- \|\nabla e^n\|^2 \leq C\left(\|\nabla V^n\|_\infty \|U^{n-1} - u(t_n)\|\|\nabla e^n\| + \|U^{n-1} - u(t_n)\|\|e^n\|\right).
	\end{equation}
	Due to Corollary \ref{cor:GradientBound} we have that $\|\nabla V^n\|_\infty \leq C$ and by using Poincar\'e-Friedrichs inequality we can bound $\|e_n\| \le C \|\nabla e_n\|$. Thus, we further obtain
	\begin{equation}
		(\partial_h^\alpha e^n, e^n) + D_- \|\nabla e^n\|^2 \leq C \|U^{n-1} - u(t_n)\| \|\nabla e^n\| \leq D_- \|\nabla e^n\|^2 + C \|U^{n-1} - u(t_n)\|^2,
	\end{equation}
	where we used the $\epsilon$-Cauchy inequality $ab \le \frac{\veps}{2} a^2 + \frac{1}{2\veps} b^2$, with an appropriate choice of $\veps$, to extract the the gradient term. Finally, by our Proposition \ref{prop:CQInner} 
	\begin{equation}
		\partial^\alpha_h \|e^n\|^2 \leq C \|U^{n-1} - u(t_n)\|^2 \leq C \left(\|U^{n-1} - V^{n-1}\| + \|V^{n-1} - u(t_{n-1})\| + \|u(t_{n-1}) - u(t_n)\|\right)^2.
	\end{equation}
	The first term on the right-hand side is the error $e^{n-1}$, the second one is $\eta^{n-1}$ and can be estimated by Theorem~\ref{thm:Auxiliary}, while the third one follows from the regularity assumption \eqref{eqn:RegularityAssumtion} 
	\begin{equation}\label{eqn:ConvergenceError0}
		\|u(t_{n-1}) - u(t_n)\| \leq C(u) t_n^{\alpha-1} h, \quad n\geq 1.
	\end{equation}
	Therefore, by the elementary inequality $(a+b)^2 \leq 2(a^2 + b^2)$ we have
	\begin{equation}\label{eqn:ErrorBound}
		\partial^\alpha_h \|e^n\|^2 \leq C \left(\|e^{n-1}\|^2 + B^2(h) t_n^{2(\alpha-1)} + k^4 + h^{2} t_n^{2(\alpha-1)} \right),
	\end{equation}
	and we are ready to apply the discrete Gr\"onwall inequality. The discrete fractional integral of $t_n^{2(\alpha-1)}$ can be bounded by, see for instance \cite[Lemma 5.3]{Lu88I} and proof of Lemma~\ref{lem:ABError}, 
	\begin{equation}\label{dfitnpow}
		h^{\alpha} \sum_{j=0}^{n-1} (n-j)^{\alpha-1} t_{j+1}^{2(\alpha-1)} = h^{3\alpha-2} \sum_{j=0}^{n-1} (n-j)^{\alpha-1} (j+1)^{2(\alpha-1)} \le C \left\{ \begin{array}{ll}  h^{3\alpha-2} n^{3\alpha-2}, \quad & \alpha >\frac{1}{2} \vspace{4pt}\\
			h^{\alpha-1} n^{\alpha-1} \ln n,  \quad & \alpha= \frac{1}{2}
			\vspace{4pt}\\
			h^{3\alpha-2} n^{\alpha-1}, \quad & \alpha < \frac{1}{2} 
		\end{array}\right..
	\end{equation}
	Therefore,
	\[
	h^{\alpha} \sum_{j=0}^{n-1} (n-j)^{\alpha-1} t_{j+1}^{2(\alpha-1)} \le C \left\{ \begin{array}{ll}  t_n^{3\alpha-2} \le  T^{2\alpha-1}  t_n^{\alpha-1}, \quad & \alpha >\frac{1}{2} \vspace{4pt}\\
		\ln h^{-1} t_n^{\alpha-1} ,  \quad & \alpha= \frac{1}{2}
		\vspace{4pt}\\
		h^{2\alpha-1} t_n^{\alpha-1}, \quad & \alpha < \frac{1}{2} 
	\end{array}\right.
	\]
	This leads to the application of Lemma \ref{lem:Gronwall} to $\|e^n\|^2$ with $\tilde{A}(h)=k^4$ and
	\[
	\tilde{B}(h) = \left\{\begin{array}{ll} (B(h)^2+h^2) , \quad & \alpha >\frac{1}{2} \vspace{4pt}\\
		(B(h)^2+h^2)  \ln h^{-1} ,   \quad & \alpha= \frac{1}{2}
		\vspace{4pt}\\
		(B(h)^2+h^2)  h^{2\alpha-1} , \quad & \alpha < \frac{1}{2} 
	\end{array}\right..
	\]
	Then $\|e^n\|^2 \le C(\tilde{A}(h)+\tilde{B}(h)t_n^{\alpha-1})$, or
	\[
	\|e^n\| \le  C \left\{ \begin{array}{ll} (B(h) + h) t_n^{(\alpha-1)/2} +k^2 , \quad & \alpha >\frac{1}{2} \\
		(B(h)+h)  (\ln h^{-1})^{1/2} t_n^{(\alpha-1)/2} + k^2 ,  \quad & \alpha= \frac{1}{2}
		\\
		(B(h)+h) h^{\alpha-1/2} t_n^{(\alpha-1)/2} + k^2, \quad & \alpha < \frac{1}{2} 
	\end{array}\right..
	\]
	We have thus obtained the bound for the term $\|e^n\|$ in the error decomposition \eqref{eqn:ErrorDecomposition} which, by recalling the auxiliary estimate \eqref{eqn:EtaBound} finishes the proof. 
\end{proof}

\begin{rem}
	Estimates \eqref{eqn:ConvergenceError} of Theorem \eqref{thm:Convergence} are \emph{globally optimal} in time for all values of $0<\alpha<1$ in the sense that they reduce to the known results for linear equations near $t_n \approx 0$ (the worst case). Moreover, these estimates are also \emph{pointwise optimal} in time for $\alpha > 1/2$ since for a fixed time $t_n > 0$ we obtain an estimate $O(h)$ as $h\rightarrow 0$ (with a logarithmic correction for $p = 1$). This, again apart from the logarithmic term, holds true for $\alpha=1/2$. For the case $\alpha<1/2$ we obtain a pointwise estimate $O(h^{\alpha+1/2})$ as $h\rightarrow 0$ which is still better than the global order $\alpha$. As is well-known, the lower regularity of the solution at the origin causes loss of accuracy of the numerical schemes. Hence, choosing a CQ based on BDF$(p)$ with $p=1,2$ is reasonable in practice. 
	
	As mentioned above, the bounds for the quasilinear problem are optimal only in the case $\alpha \geq 1/2$. The reason for this is the specific proof technique in which we estimate the square of the error as in \eqref{eqn:ErrorBound}. This arises due to interplay between the nature of the energy method, the solution-dependent diffusion coefficient, and a subtle way in which the truncation error behaves. Consequently, we have to use the $\epsilon$-Cauchy inequality to cancel the gradient term leading to the square estimate of the error and the estimates of the sum \eqref{dfitnpow}. This phenomenon is not present in the semilinear case. We note that the lack of the optimal bounds for $\alpha < 1/2$ is the effect of the proof technique and using a different approach would probably give pointwise optimal estimates for all $0<\alpha<1$ (compare the operator-perturbation technique of \cite{JinQuanWohlZhou24} used under some more demanding regularity assumptions than these in the present paper). However, the advantage of using the energy method provides a self-contained proof under relatively weak regularity assumptions. 
\end{rem}

\section{Fast and oblivious implementation}
To implement our numerical scheme \eqref{eqn:NumericalScheme} , we fix a basis of the $V_k$ space and expand the solution $U^n$, that is
\begin{equation}
	U^n = \sum_{k=1}^M y_k^n \Phi_k.
\end{equation}

Taking $\chi = \Phi_j$, denoting $\textbf{y}^n = (y_1^n,...,y_M^n)$, and plugging the above into \eqref{eqn:NumericalScheme} we obtain the following
\begin{equation}
	B \partial_h^\alpha \textbf{y}^n + A(\textbf{y}^{n-1}) \textbf{y}^n = \textbf{f}^n(\textbf{y}^{n-1}),
\end{equation}
where the mass matrix $B = \left\{B_{ij}\right\}_{i,j=1}^M$, the stiffness matrix $A = \left\{A_{ij}\right\}_{i,j=1}^M$, and the load vector $\textbf{f}^n=\left\{f_i^n\right\}_{i=1}^M$ are defined by
\begin{equation}
	B_{ij} = (\Phi_i, \Phi_j), \quad A_{ij}(\textbf{y}^{n-1}) = a\left(D\left(\sum_{k=1}^M y_k^{n-1} \Phi_k\right); \Phi_i,\Phi_j\right), \quad f_i^{n-1} = \left(f\left(t_n, \sum_{k=1}^M y_k^{n-1} \Phi_k\right), \Phi_i\right).
\end{equation}
Finally, we discretize the Caputo derivative according to the CQ scheme \eqref{eqn:CQCaputo} to arrive at a \emph{linear system} of algebraic equations
\begin{equation}
	\label{eqn:GalerkinImplementation}
	\left(w_0 B + A(\textbf{y}^{n-1})\right) \textbf{y}^n = -B\sum_{j=0}^{n-1} w_{n-j} \textbf{y}^j + \textbf{f}^n(\textbf{y}^{n-1}),
\end{equation}
which clearly indicates the non locality in time: the right-hand side depends on the historical values of the solution $\textbf{y}^i$ for $i=0,1,...,n-1$. As a simple example of the basis, in one spatial dimension we can have $\Omega = (0,1)$ for which we can take the usual tent functions
\begin{equation}
	\Phi_j(x) = 
	\begin{cases}
		1-\left|\dfrac{x-x_j}{k}\right|, & |x-x_j|\leq k,\\
		0, & |x-x_j| > k.\\
	\end{cases}
\end{equation}
However, we do not implement directly \eqref{eqn:GalerkinImplementation}. Instead, we apply the fast algorithm developed in \cite{BanLo19} for the evaluation of the fractional integral, in order to deal with memory more efficiently and reduce computational cost. To do this, we use the preservation of the composition rule by all CQ schemes, which implies
\[
\partial_h^\alpha \textbf{y}^n  = \partial_h^{\alpha-1} \partial_h^{1} \textbf{y}^n. 
\]
In the case of the Euler based CQ, this yields
\[
\partial_h^\alpha \textbf{y}^n  = \partial_h^{\alpha-1} \textbf{v}^n, \quad \mbox{with } \quad  \textbf{v}^n:=\frac{\textbf{y}^n-\textbf{y}^{n-1}}{h}. 
\]
Setting $\widetilde{\omega}_j$ the Euler based CQ weights associated with the fractional integral $\partial_h^{\alpha-1}$, for $j=0,\dots, N$, we obtain from \eqref{eqn:CQCaputo} 
\begin{equation}
	\label{eqn:fastCQscheme}
	B \sum_{j=0}^{n} \widetilde{\omega}_{n-j} \textbf{v}_j  + A(\textbf{y}^{n-1}) \textbf{y}^n = \textbf{f}^n(\textbf{y}^{n-1}),
\end{equation}
which leads to the linear system
\begin{equation}\label{linsys}
	\left( \frac{1}{h}B\widetilde{\omega}_0 + A(\textbf{y}^{n-1}) \right) \textbf{y}^{n} = \frac{1}{h} \widetilde{\omega}_0 B \textbf{y}^{n-1} + \textbf{f}^n(\textbf{y}^{n-1})- B\sum_{j=0}^{n-1} \widetilde{\omega}_{n-j} \textbf{v}_j. 
\end{equation}
The computation of the memory term on the right-hand side requires in principle the precomputation and storage of all the CQ weights $\widetilde{\omega}_j$, $j=0,\dots,N$, the storage of all vectors $\textbf{v}_j$, with $j=0,\dots,n-1$, and $2n$ operations per time step, implying a total complexity growing like $O(N^2)$, as $N\to \infty$. The algorithm in \cite{BanLo19} uses a special integral representation and quadrature to compute the CQ weights within a prescribed precision $\varepsilon$ and manages to reduce the complexity to $O(N\log{N}\log(\frac{1}{\varepsilon}))$. Moreover, if we are only interested in the solution at the final time $T$, the memory in the evaluation of the right-hand side grows like $O(\log{N}\log(\frac{1}{\varepsilon}))$, since the algorithm does not need to store the entire history $\textbf{v}_j$, for all $j< n$, and all the CQ weights. Instead, for a moderate value of $n_0$, such as $n_0=4,5$, only the first $n_0+1$ CQ weights $ \widetilde{\omega}_{j}$ and the last $n_0+2$ values of $\textbf{v}_{n-j}$, with $0\le j \le n_0+2$, are required in storage, together with $O(\log{N}\log(\frac{1}{\varepsilon}))$ linear combinations of the history $\textbf{v}_{n-j}$, with $j>n_0+1$. Although the main part of the computational cost of the method \eqref{linsys} lays in the assembly of the stiffness matrix $A(\textbf{y}^{n-1})$ at every time step, we can see from the results reported in the next Section that the application of the algorithm in \cite{BanLo19} is advantageous from the complexity point of view. CPU times are globally reduced by a factor of almost two, as shown in Figure~\ref{fig:TimeRatio}. Since in the present paper we are mostly interested in the verification of our error estimates, both pointwise and uniform in time, we compute and store the numerical solution $y^n$ for every $1\le n \le N$ and do not use an oblivious version of the algorithm. 

\section{Numerical examples}
We will illustrate our results with some numerical examples. For illustration purposes, we take $\Omega = (0,1)$ as our spatial domain and focus only on the temporal features of the solved solution. The discretization in space has been implemented as the finite element method described in the previous section. The grid has been taken to be sufficiently fine to be able to neglect all errors of spatial nature when compared to temporal discretization. The discretization of the fractional derivative is done using BDF$(p)$-CQ with $p = 1,2$.  

We consider two exemplary problems: one will be chosen artificially in order to have an exact solution
\begin{equation}\label{eqn:NumExampleExact}
	\begin{split}
		D(x,t,u) &= e^{-u}, \quad f(x,t,u) = e^{-u}t^\alpha \left(2+t^\alpha\left(1-2x\right)^2\right)+\Gamma (1+\alpha ) x(1-x), \\ 
		u(x,t) &= t^\alpha x(1-x), 
	\end{split}
\end{equation}
which models the typical solution's behavior at $t=0$. The second example is a more realistic model of a porous medium with a simple concentrated source. The diffusivity in porous media strongly depends on the moisture content (for example, see \cite{gardner1959solutions}) and exponential model is one of the typical choices
\begin{equation}\label{eqn:NumExampleRealistic}
	D(x,t,u) = e^{-u}, \quad f(x,t,u) = \frac{1}{\sqrt{4\pi \delta}} e^{-\frac{(x-x_0)^2}{4\delta}}, \quad u(x,0) = 0.
\end{equation}
The parameters $k$, $x_0$, and $\delta$ are chosen accordingly for a particular simulation. Note that in this case, we do not possess an exact analytic solution.

In the first example \eqref{eqn:NumExampleExact} we can compare the numerical solution with the exact one and compute the error at the first step $t=h$ and the final simulation time $t = T$, that is, for each $\alpha$ we find $
\|u(t_n) - U^n\|$, where $t_n = h$ and $t_n = T$. These choices are made in order to investigate the global and pointwise error estimates.  To eliminate the pollution of the spatial discretization error, we performed the experiments with a rather fine spatial grid, with 100 degrees of freedom. The results of our calculations are presented in Figs. \ref{fig:errorExact} and \ref{fig:errorExacth}. As we can see, numerical computations verify that the scheme is convergent even for the nonsmooth in time case. The observed order of convergence, measured at $t=T$ is  consistent with $1$ for all $\alpha$ with a slight deviation for its small values.  The reason for this might be that for $\alpha$ close to $0$ the time evolution is very slow and it is difficult to assess the error of the scheme. This behavior does not contradict our theory, which predicts the order of convergence 0.8 and 0.6 for $\alpha=0.3$ and $\alpha=0.1$, respectively.  Also, due to the lack of sufficient smoothness of the solution, BDF2-CQ performs essentially the same way. Therefore, there is no real gain in using higher-order quadrature for our scheme. We can conclude that numerical computations verify our theoretical findings of Theorem \ref{thm:Convergence}.

For the error measured at $t_1 = h$ (see Fig. \ref{fig:errorExacth}) we observe at most order of convergence $\alpha$. We notice that we are almost always violating the CFL condition \eqref{cfld1} at $t_1$, since the spatial discretization is set with $k=1/40$ and thus there is no contradiction with Theorem~\ref{thm:Convergence}.  Calculating the tangent of the error lines indicates that we observe convergence of order between $\alpha$ and $\alpha/2$. The same conclusion is true for BDF2-CQ, where we cannot hope for higher-order accuracy due to the limited regularity of the solution and the semi-implicit character of our scheme, both in the diffusion coefficient and the source. %This result is consistent with similar schemes for linear and semilinear equations. 

\begin{figure}
	\centering
	\includegraphics[trim = 150 270 100 270, scale = 0.72]{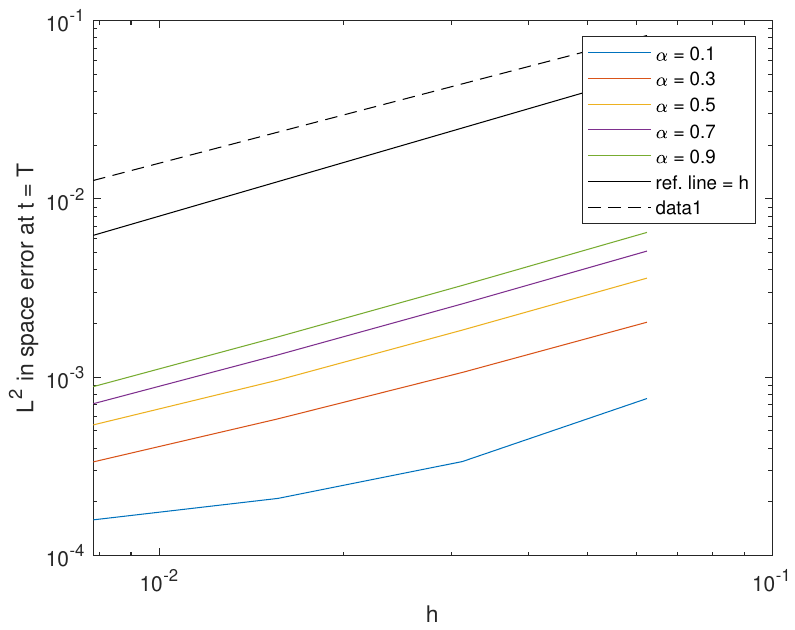}
	\includegraphics[trim = 110 270 150 270, scale = 0.72]{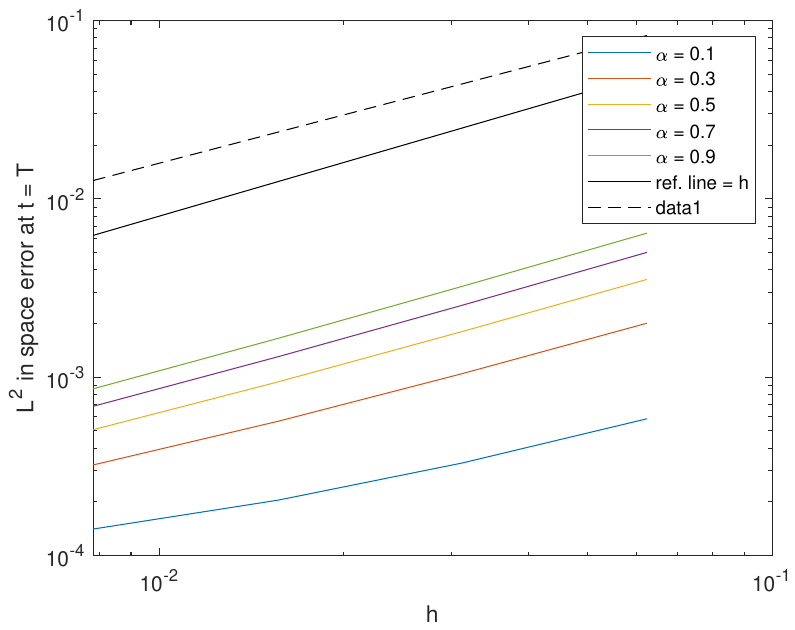}
	\caption{Error $\|u(T) - U^n\|$ with $nh = T = 1$ for the problem \eqref{eqn:NumExampleExact}. On the left: Euler CQ. On the right: BDF2-CQ. }
	\label{fig:errorExact}
\end{figure}

\begin{figure}
	\centering
	\includegraphics[trim = 150 270 100 270, scale = 0.72]{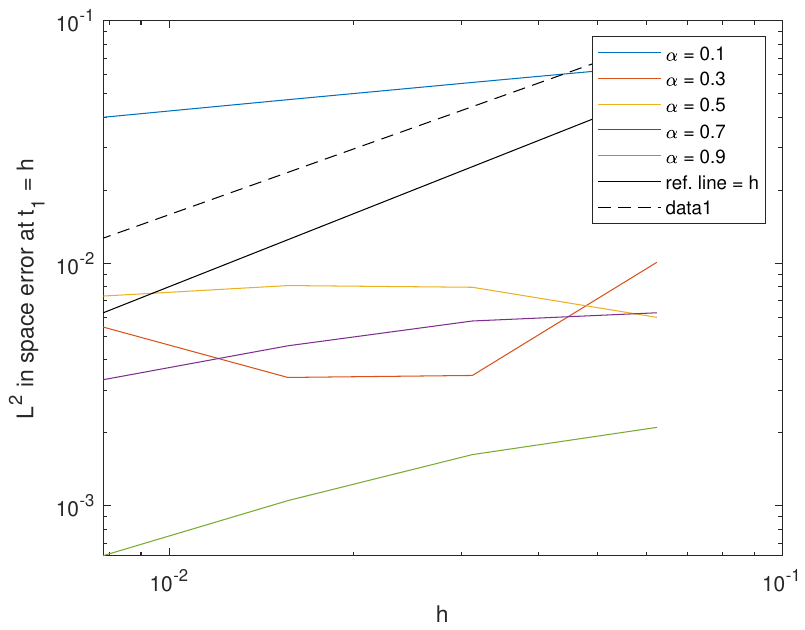}
	\includegraphics[trim = 110 270 150 270, scale = 0.72]{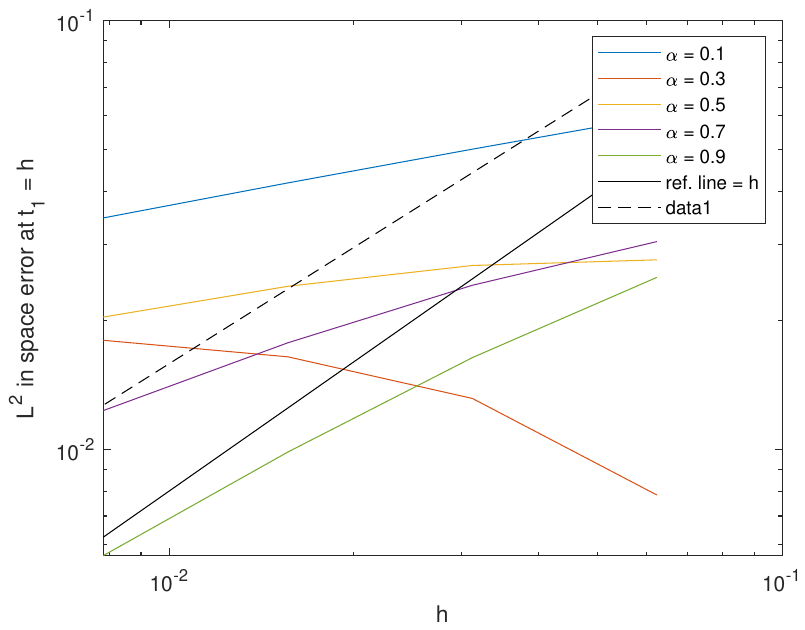}
	\caption{Error $\|u(h) - U^1\|$ for the problem \eqref{eqn:NumExampleExact}. On the left: Euler CQ. On the right: BDF2-CQ. }
	\label{fig:errorExacth}
\end{figure}

For the second example \eqref{eqn:NumExampleRealistic} the error cannot be computed directly and we will estimate the order of convergence by the Aitken extrapolation \cite{Lin85}. Assume that the error can be estimated with $\max_n\|U^n_h - U^{n}_{h/2}\|\approx C h^p$, where as a reference solution we take the one computed on a twice finer grid. Then, by halving the grid once more and taking the logarithm, we can write
\begin{equation}\label{eqn:Aitken}
	p \approx \log_2\frac{\|U^n_h - U^{n}_{h/2}\|}{\|U^n_{h/2} - U^{n}_{h/4}\|}, \quad \text{with fixed} \quad n h = t_n = T. 
\end{equation}
That is, the order is estimated pointwise in time and in the $L^2$ norm in space. The results of our computations are gathered in Tab. \ref{tab:EstOrder}. As we can see, the estimated order is close to $1$ for all values of $\alpha$, again consistent with the Euler discretization measured pointwise in time. The BDF2-CQ scheme also converges, and its empirical order is consistent with $1$. Again, because of the low regularity of the solution, using higher-order quadrature does not necessarily lead to improved accuracy. 

The numerically estimated orders of convergence show that our scheme for this example performs as theoretically anticipated in Theorem \ref{thm:Convergence}: globally for all $0<\alpha<1$ and pointwise at least for $\alpha \geq 1/2$.

\begin{table}
	\centering
	\begin{tabular}{cccccccccc}
		\toprule
		$\alpha$ & 0.1 &  0.2 & 0.3 & 0.4 & 0.5 & 0.6 & 0.7 & 0.8 & 0.9 \\
		\midrule 
		Euler & 0.77 & 0.90 & 1.06 & 1.08 & 0.93 & 0.93 & 1.03 & 0.93 & 1.11 \\
		BDF2 & 0.99 & 0.97 & 0.95 & 0.91 & 0.87 & 0.92 & 0.97 & 0.92 & 0.87 \\
		\bottomrule
	\end{tabular}
	\caption{Estimated order of convergence for the second example \eqref{eqn:NumExampleRealistic} for different $\alpha$ and quadratures. The basis of our calculation has been taken $h = 2^{-8}$ for the formula \eqref{eqn:Aitken} and $T=1$. }
	\label{tab:EstOrder}
\end{table}

The final example concerns the temporal complexity of our algorithm. We have compared the computation times of three ways of implementing the time integration of our PDE: with and without the fast and oblivious algorithm described in the previous section, and the L1 scheme. In Fig. \ref{fig:TimeRatio} we can see the following ratio computed for different values of $\alpha$
\begin{equation}
	\frac{\text{computation time of a standard or L1 implementation}}{\text{computation time of the fast and obliovious implementation}}.
\end{equation}
The problem tested is our second example \eqref{eqn:NumExampleRealistic}. In our calculations, we have taken $h = 2^{-9}$ but have also tested other values. In addition, to obtain Fig. \ref{fig:TimeRatio} independently of various computer background processes, we have performed simulations $100$ times and taken the mean values. The results are uniform with respect to $\alpha$ (note the vertical scale) and indicate that the fast and oblivious implementation is on average twice as fast as the standard implementation. 

\begin{figure}
	\centering
	\includegraphics[trim = 150 270 150 270, scale = 0.9]{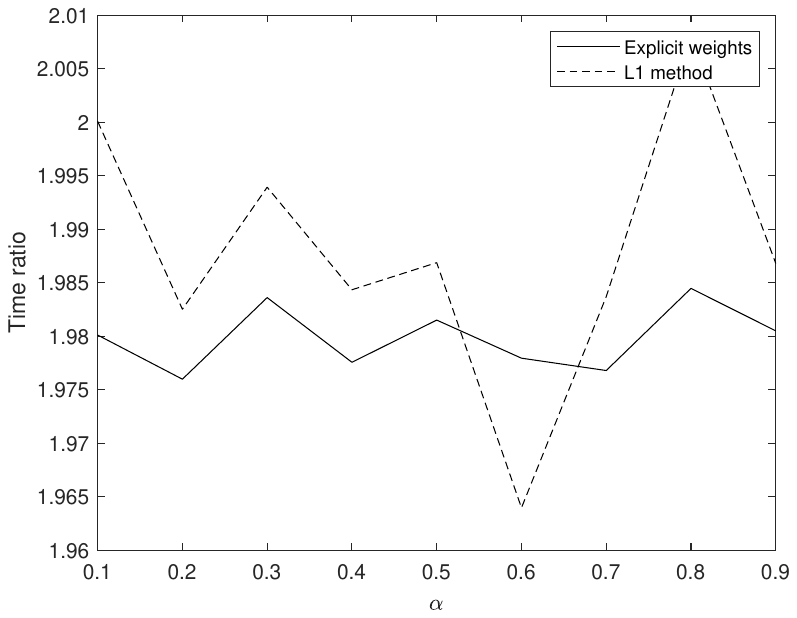}
	\caption{Mean ratio of calculation times: without and with the implementation of the fast and oblivious algorithm. The comparison with the L1 scheme is also presented. }
	\label{fig:TimeRatio}
\end{figure}

\section{Conclusion}
The Convolution Quadrature can be applied to the quasilinear subdiffusion equation yielding a convergent scheme for quadratures satisfying \eqref{eqn:Assumptions}. When supplied with fast and oblivious implementation, the computation time can be reduced by at least twice, which is much desired in the time-fractional setting. Numerical calculations verify our theoretical error estimates that are globally optimal for all $0<\alpha<1$ and pointwise for $\alpha \geq 1/2$. Future work includes suggesting some initial corrections to the CQ scheme and finding optimal estimates for the case of small $\alpha$. 

\section*{Acknowledgement}
The first author has been supported by the ``Beca Leonardo for Researchers and Cultural Creators 2022'' granted by the BBVA Foundation, by the ``Proyecto 16 - proyecto G Plan Propio" of the University of Malaga, and by grant PID2022-137637NB-C21 funded by MICIU/AEI/10.13039/501100011033 and ERDF/EU.

Ł.P. has been supported by the National Science Centre, Poland (NCN) under the grant Sonata Bis with a number NCN 2020/38/E/ST1/00153.

\end{document}